\documentclass[reqno,11pt]{amsart}
\usepackage{amsmath,amssymb,amsfonts,amsthm,amscd,indentfirst,thmtools,mathtools,mathrsfs}

\usepackage[hyphens]{url}
\usepackage{hyperref}
\usepackage{bookmark}
\mathtoolsset{showonlyrefs=true}

\newcounter{mparcnt}

\usepackage{fancyhdr}
\usepackage{esint}
\usepackage{enumerate}

\usepackage{pictexwd,dcpic}
\usepackage{graphicx}
\usepackage{caption}
\usepackage{xcolor}
\usepackage{enumitem}

\usepackage[msc-links]{amsrefs} %--reference with MR, MATHSCI format--

\usepackage{epsfig,here}
\usepackage{subfigure,here}

\newtheorem{theorem}{Theorem}[section]%[section]
\newtheorem{lemma}[theorem]{Lemma}%[section]
\newtheorem{proposition}[theorem]{Proposition}%[section]
\newtheorem{remark}[theorem]{Remark}

\usepackage[top=4cm, bottom=4cm, right=3cm, left=3cm]{geometry}

\def\om{\omega}
\def\Om{\Omega}
\def\p{\partial}
\def\ep{\epsilon}
\def\de{\delta}
\def\De{\Delta}

\def\S{{\Sigma}}
\def\<{\langle}
\def\>{\rangle}
\def\div{{\rm div}}
\def\na{\nabla}

\providecommand{\abs}[1]{\lvert#1\rvert}
\providecommand{\Abs}[1]{\left\lvert#1\right\rvert}

\providecommand{\norm}[1]{\lVert#1\rVert}

\newcommand{\mbN}{\mathbb{N}}

\newcommand{\mbR}{\mathbf{R}}
\newcommand{\mbS}{\mathbf{S}}

\newcommand{\mcA}{\mathcal{A}}

\newcommand{\mcH}{\mathcal{H}}

\newcommand{\mcL}{\mathcal{L}}

\newcommand{\mcW}{\mathcal{W}}

\newcommand{\mfR}{\mathbf{R}}

\newcommand{\rd}{{\rm d}}

\newcommand{\pr}{\partial_{\rm rel}}

\newcommand{\ra}{\rightarrow}

\newcommand{\eq}[1]{\begin{equation}\begin{alignedat}{2} #1 \end{alignedat}\end{equation}}

\numberwithin{equation}{section}

\begin{document}
	
\title[Quantitative Alexandrov theorem]{Quantitative Alexandrov theorem for capillary hypersurfaces in the half-space}

\date{\today}

\author[Jia]{Xiaohan Jia}
\address[X.J]{School of Mathematics\\
Southeast University\\
 211189, Nanjing, P.R. China}
\email{xhjia@seu.edu.cn}
	
\author[Zhang]{Xuwen Zhang}
\address[X.Z]{Mathematisches Institut\\
		Universit\"at Freiburg\\
	Ernst-Zermelo-Str.1\\
		79104\\
  \newline\indent Freiburg\\ Germany}
\email{xuwen.zhang@math.uni-freiburg.de}

%\thanks{This work is  supported by test.
%}

\dedicatory{Dedicated to Professor Guofang Wang on the occasion of his 60th birthday}

\begin{abstract}

In this paper, we prove the quantitative version of the Alexandrov theorem for capillary hypersurfaces in the half-space,  which generalizes Julin-Niinikoski's result \cite{JN23} to the capillary case.
The proof is based on the quantitative analysis of the Montiel-Ros-type argument carried out in our joint works with Wang-Xia \cites{JWXZ22-B,JWXZ23}.
\
		
\\
\end{abstract}
	
\maketitle
\tableofcontents
%=========
\section{Introduction}

Constant mean curvature (CMC) hypersurfaces play the role as stationary points of the isoperimetric problem.
A celebrated theorem in differential geometry, which is known as \textit{Alexandrov's soap bubble theorem} \cite{Aleksandrov62}, says that any bounded domain $\Om$ whose boundary $\p\Om$ is a smooth, connected hypersurface in $\mfR^{n+1}$ with constant mean curvature, is a geodesic ball.

%--------
\subsection{Quantitative soap bubble problem}
A natural question is to characterize the shape of any such domain $\Om$, when the mean curvature of $\p\Om$, say $H_{\p\Om}$, is a perturbation from some constant.
This is the so-called \textit{quantitative soap bubble problem}, and has been intensively studied by many mathematicians, here we mention some of the works in this direction.

In \cite{CV18}, by using a quantitative treatment of the moving plane method, Ciraolo-Vezzoni showed that if the oscillation of $H_{\p\Om}$ is small enough, $\p\Om$ then stays between two concentric spheres, with difference of radii controlled by $C\cdot{\rm osc}(H_{\p\Om})$, where $C$ depends only on dimensional constant $n$, area $\abs{\p\Om}$, and lower bound for interior and exterior balls, the rate of stability is sharp, as can be seen from a simple calculation for ellipsoids.
In \cite{CM17}, under the assumption that $\p\Om$ is mean convex, Ciraolo-Maggi proved the following quantitative stability result: when the so-called \textit{Alexandrov deficit}
\eq{
\de(\Om)
\coloneqq\frac{\norm{H_{\p\Om}-H_0}_{C^0(\p\Om)}}{H_\Om},\text{ where }H_0=\frac{n\abs{\p\Om}}{(n+1)\abs{\Om}}
}
is small, $\p\Om$ will be close to a family of controllable number of disjoint spheres with equal radii, in the sense of Hausdorff distance, while the closeness is measured exponentially with coefficient $C$ depends only on $n,\abs{\p\Om},{\rm diam}(\Om)$, and the exponential power is given by $\alpha=\frac1{2(n+2)}$.
The fact that $C$ does not depend on the interior and exterior radius accounts for the so-called \textit{bubbling phenomenon},
see also \cite{DMMN18} for some developments in the anisotropic frameworks.
For some recent progress on the quantitative soup bubble problem involving interior and exterior radius, see \cites{MP19,MP20-1,MP20-2,Scheuer21,SX22,SZ23,Poggesi24}.
Assuming the strictly convexity of $\Om$,
some previous stability results can be found in the references in \cites{CM17,CV18}.

Quite recently, beyond mean convexity, Julin-Niinikoski \cite{JN23} showed the following quantitative result: provided that the \textit{$L^n$-deficit} with respect to some positive constant $\lambda$,
\eq{
\norm{H_{\p\Om}-\lambda}_{L^n(\p\Om)},
}
is small, $\p\Om$ should be close to a family of controllable number of disjoint spheres with equal radii $\frac{n}\lambda$, in the sense that the Hausdorff distance will be controlled exponentially by the $L^n$-deficit with coefficient depends only on $n$, the upper bound of $\abs{\p\Om}$ and the lower bound of $\abs{\Om}$, while the exponential power is given explicitly by some dimensional constant $q(n)\in(0,1]$.
Their approach is based on a subtle quantitative analysis of the \textit{Montiel-Ros argument} \cite{MR91}.
%-------
\subsection{Capillary Alexandrov's theorem}
Our main focus of the paper is the Alexandrov's theorem for capillary hypersurfaces in the half-space $\mfR^{n+1}_+=\{x\in\mfR^{n+1}:\left<x,E_{n+1}\right>>0\}$.
For the long history of the study of capillary surfaces, we refer to the monograph \cite{Finn86} and also \cite{Mag12}*{Chapter 19} for an overview.
A well-known result is the classical \textit{isoperimetric capillarity problem} in the half-space, see e.g., \cite{Mag12}*{Theorem 19.21}, with the quantitative version recently shown by Pascale-Pozzetta in \cite{PP24}.

Throughout the paper, we use the terminology \textit{$\theta$-capillary hypersurfaces} in the half-space to refer to embedded, compact, $C^2$-hypersurfaces that are supported on $\p\mfR^{n+1}_+$ and intersecting $\p\mfR^{n+1}_+$ with a fixed contact angle $\theta\in(0,\pi)$.
The Alexandrov's theorem for capillary hypersurfaces in the half-space states that:
\begin{center}{
\it 
The $\theta$-cap is the only
connected, CMC, $\theta$-capillary hypersurface in $\overline{\mfR^{n+1}_+}$.
}
\end{center}
The first proof of this result goes back to \cite{Wente80}, in which the moving plane method together with Serrin's boundary point lemma \cite{Serrin71} is used by Wente to show the axially symmetric of $\S$ with respect to some vertical direction of $\p\mfR^{n+1}_+$.
Recently, joint with Wang-Xia we reprove the Alexandrov-type theorem by using the so-called \textit{integral method} initiated by Ros \cite{Ros87} in \cite{JXZ23}; and by using the Montiel-Ros-type argument \cite{MR91} in \cite{JWXZ22-B}.
The intermediate step of both approaches is establishing the \textit{Heintze-Karcher-type inequalities} for mean convex capillary hypersurfaces. 
The former approach is based on the study of a specific torsion potential problem on $\Om\subset\mfR^{n+1}_+$, a bounded domain that is adhering to $\p\mfR^{n+1}_+$, whose relative boundary $\S=\overline{\p\Om\cap\mfR^{n+1}_+}$ is a $\theta$-capillary hypersurface.
Precisely, in \cite{JXZ23} the solution to the following mixed boundary elliptic equation is considered:
\eq{
\begin{cases}
\De f=1\quad&\text{ in }\Om,\\
f=0\quad&\text{ on }\S,\\
\left<\na f,E_{n+1}\right>=\frac{n}{n+1}\cot\theta\frac{\abs{\p\Om\cap\p\mfR^{n+1}_+}}{\abs{\p\S}}\quad&\text{ on }\p\Om\cap\p\mfR^{n+1}_+.
\end{cases}
}
The key ingredient is the \textit{Reilly-type formula} \cites{Reilly77,QX15,LX19}, which relates the trace-free Hessian of $f$ to the mean curvature $H_\S$.
However, the regularity of $f$ becomes a delicate issue when using Reilly's formula, which is why the Alexandrov-type theorem, together with a quantitative version in \cite{JLXZ23}, can be proved only for $\theta\in(0,\frac\pi2]$.

The later approach \cite{JWXZ22-B}, based on a purely geometric argument enlightened by Montiel-Ros \cite{MR91}, perfectly solves the concern and completes the whole range of $\theta$.
Let us quickly review this argument:

For any $\theta\in(0,\pi)$ and any bounded domain $\Om\subset \overline{\mfR^{n+1}_+}$ whose relative boundary $\S$ is a mean convex $\theta$-capillary hypersurface in $\overline{\mfR^{n+1}_+}$, we define a set
\eq{
Z=\left\{(x,t)\in \S\times \mfR: 0<t\leq \frac{1}{\max_i \{\kappa_i(x)\}}\right\},
}
and a map  which indicates a family of \textit{shifted parallel hypersurfaces},
\eq{
\zeta_\theta: Z\to \mfR^{n+1}: \quad \zeta_\theta(x, t)=x-t(\nu-\cos\theta E_{n+1})
}
where $\kappa_i, i=1,\ldots,n,$ are the principal curvatures and $\max_i \kappa_i(x)\geq\frac{H_\S}{n}>0$.
Using the capillary boundary condition, one finds that $\zeta_\theta$ is surjective onto $\Omega$, namely, $\Om\subset \zeta_\theta(Z)$. 
By the area formula and the AM-GM inequality, one obtains
\eq{\label{ineq-hk-iso-argument}
    \abs{\Om}
    \le\abs{\zeta_\theta(Z)}
    &\leq\int_{\zeta_\theta(Z)}\mcH^0(\zeta_\theta^{-1}(y))\rd y=\int_Z{\rm J}^Z\zeta_\theta\rd\mcH^{n+1}\\
    &=\int_{\S}\int_0^\frac{1}{\max_i \kappa_i(x)}(1-\cos\theta\left<\nu, E_{n+1}\right>)\prod_{i=1}^n(1-t\kappa_i(x))\rd t\rd\mcH^n(x)\\
    &\leq\int_\S(1-\cos\theta\left<\nu, E_{n+1}\right>)\int_0^\frac{1}{\max_i \kappa_i(x)}\left(1-t\frac{H_\S(x)}{n}\right)^n\rd t\rd\mcH^n\\
    &\leq\int_\S(1-\cos\theta\left<\nu, E_{n+1}\right>)\int_0^\frac{n}{H_\S(x)}\left(1-t\frac{H_\S(x)}{n}\right)^n\rd t\rd\mcH^n\\
    &\le\frac{n}{n+1}\int_{\S}\frac{1-\cos\theta\left<\nu,E_{n+1}\right>}{H_\S}\rd\mcH^n,
}
so that the Heintze-Karcher inequality holds: 
\eq{\label{ineq-hk-iso}
\abs{\Om}
\le\frac{n}{n+1}\int_{\S}\frac{1-\cos\theta\left<\nu,E_{n+1}\right>}{H_\S}\rd\mcH^n
}
with equality achieved if and only if $\Om$ is a \textit{$\theta$-ball} which is adhering to $\p\mfR^{n+1}_+$, characterized by $B_{r;\theta}(o)\cap\mfR^{n+1}_+$, where $r>0,o\in\p\mfR^{n+1}_+$ and
\eq{
B_{r;\theta}(o)
\coloneqq\{x\in\mfR^{n+1}:\Abs{x-(o-r\cos\theta E_{n+1})}\leq r\}.
}
The relative boundary of $\theta$-balls in $\mfR^{n+1}_+$ is called \textit{$\theta$-caps}.

Combining \eqref{ineq-hk-iso} with the Minkowski-type formula (see e.g., \cite{AS16})
\eq{
\int_{\S}\left<x,H_\S\nu\right>\rd\mcH^n
=\int_{\S}n\left(1-\cos\theta\left<\nu, E_{n+1}\right>\right)\rd\mcH^n,
}
and also the elliptic point lemma \cite{JWXZ22-B}*{Proposition 5.2}, one finds:
if $\S$ is CMC, then equality in \eqref{ineq-hk-iso} holds and in turn, $\Om=B_{\frac{n}{H_\S};\theta}(o)\cap\mfR^{n+1}_+$ for some $o\in\p\mfR^{n+1}_+$.

In view of the rich study of quantitative soap bubble problem for closed hypersurfaces as mentioned above, it would be interesting to ask whether or not the quantitative Alexandrov's soap bubble theorem holds for hypersurfaces with free boundary or capillary boundary.
The purpose of this paper is to address this problem in the half-space case, with suitable and natural choices of factors and coefficients which appear in the stability measurement.

%------

\subsection{Main result}

Our main result is as follows.

\begin{theorem}\label{Thm-Stability}
Given $n\in\mbN^+$, $\theta\in(0,\pi)$,
let $\S\subset\overline{\mbR^{n+1}_+}$ be a compact, embedded $\theta$-capillary hypersurface.
Let $\Om$ denote the enclosed region of $\S$ with $\p\mbR^{n+1}_+$ such that $\p\Om=\S\cup T$.
Given $\lambda\in\mbR_+$ and $1\leq C_0<\infty$, let $R=\frac{n}\lambda$.
There exist positive constants
\eq{
C=C(n,\theta,C_0),\text{ }N\leq N_0=N_0(n,\theta,C_0),
\text{ }\de=\de(n,\theta,C_0),
}
and points
\eq{
o_1,\ldots,o_N\in\overline\Om\subset\overline{\mfR^{n+1}_+},
}
such that if $P(\Om)\leq C_0$, $\abs{\Om}\geq C_0^{-1}$ and
\eq{\norm{H_\S-\lambda}_{L^n(\S)}\leq\de,
}
then for $S_\theta\coloneqq\bigcup_{i=1}^N\p B_{R;\theta}(o_i)\cap\overline{\mfR^{n+1}_+}$, there holds
\eq{\label{esti-Hausdorff-distance}
{\rm dist}(\S,S_\theta)
\leq C\norm{H_\S-\lambda}_{L^n(\S)}^{\frac1{(n+2)^2}}.
}
In particular,
for the centers $o_1,\ldots,o_N$,
either
\eq{\label{ineq-o_i,E_n+1-Thm-1}
\left<o_i,E_{n+1}\right>
\leq C\norm{H_\S-\lambda}_{L^n(\S)}^\frac1{(n+2)^2},
}
or
\eq{\label{ineq-o_i,E_n+1-Thm-2}
\left<o_i,E_{n+1}\right>
\geq (1+\cos\theta)R-C\norm{H_\S-\lambda}_{L^n(\S)}^\frac1{(n+2)^2}.
}
Moreover,
\begin{enumerate}
    \item For $\theta\in[\frac\pi2,\pi)$, and for any $1\leq i\neq j\leq N$,
\eq{\label{ineq-o_i-o_j-1}
\abs{o_i-o_j}
\geq2R-2C\norm{H_\S-\lambda}_{L^n(\S)}^{\frac1{(n+2)^2}}.
}
    \item For $\theta\in(0,\frac\pi2)$, if both $o_i,o_j$ satisfy \eqref{ineq-o_i,E_n+1-Thm-1}, then
\eq{\label{ineq-o_i-o_j-2}
F_\theta^o(o_i-o_j)
\geq2R-2C\norm{H_\S-\lambda}_{L^n(\S)}^\frac1{(n+2)^2},
}
otherwise
\eq{\label{ineq-o_i-o_j-3}
\abs{o_i-o_j}
\geq2R-2C\norm{H_\S-\lambda}_{L^n(\S)}^{\frac1{(n+2)^2}}.
}
\end{enumerate}
\end{theorem}

\begin{remark}
\normalfont
Some comments on Theorem \ref{Thm-Stability} are as follows.
\begin{enumerate}%[label=(\roman*)]
%-----
\item [($i$)]
The target shape $S_\theta$ consists of finitely many (almost) $\theta$-caps or spheres that lie completely in $\overline{\mfR^{n+1}_+}$ of equal radii, due to the estimates on height of the centers \eqref{ineq-o_i,E_n+1-Thm-1}, \eqref{ineq-o_i,E_n+1-Thm-2}.
Moreover, these $\theta$-caps and spheres are almost mutually disjoint, thanks to the
estimates of distances \eqref{ineq-o_i-o_j-1}, \eqref{ineq-o_i-o_j-2}, \eqref{ineq-o_i-o_j-3}. 
%-----
\item [($ii$)]
If the $L^n$-deficit vanishes, i.e., $H_\S=\lambda$ for $\mcH^n$-a.e. $x\in\S$, then $\S$ has to be the finite union of disjoint $\theta$-caps or spheres that lie completely in $\overline{\mfR^{n+1}_+}$ of equal radii.
%This is exactly the weak Alexandrov's theorem proved in \cite{XZ23}.
%---------
\item [($iii$)]
For the $S_\theta$ resulting from Theorem \ref{Thm-Stability}, consider now a ball with small enough radius, say $B_{\tilde\ep}$, which is disjoint from $\p\Om$ and far away from $\p\mfR^{n+1}_+$.
For the set $\widetilde\Om\coloneqq\Om\cup B_{\tilde\ep}$, it is not difficult to see that $\norm{H_\S-\lambda}_{L^p(\p B_{\tilde\ep})}$ will be sufficiently small if $1<p<n$, therefore revealing the sharpness of Theorem \ref{Thm-Stability}, in the sense that the $L^n$-deficit cannot be replaced by any weaker $L^p$-deficit. 
\end{enumerate}
\end{remark}

Further applications of our main result would be interesting to explore.
In view of \cites{DM19,XZ23} and \cites{MPS22,JN23}, it is natural to ask whether, for a large class of initial data, the global-in-time weak solutions (\textit{flat flow}) of the \textit{volume-preserving mean curvature flow},
in the capillary settings,
converge to a finite union of $\theta$-balls and Euclidean balls.

%--------
\subsection{Capillarity and anisotropy}
To prove our main result, we crucially use an idea of De Philippis-Maggi \cite{DePM15} (see also \cites{DeMasi22,LXZ23}), that half-space capillary problem is essentially an anisotropic problem, which can be seen by the definition of the following convex gauge:

Given a prescribed capillary angle $\theta\in(0,\pi)$, define the \textit{capillary gauge} $F_\theta:\mfR^{n+1}\ra\mfR^+$ by
\eq{\label{defn-F_theta}
F_\theta(\xi)
=\abs{\xi}-\cos\theta\left<\xi,E_{n+1}\right>.
}
As shown in \cite{LXZ23}*{(1.2)-(1.3)}, $\theta$-capillary hypersurfaces in the half-space are exactly anisotropic free boundary hypersurfaces (see \cites{JWXZ23,JWXZ23b} for the definition) in the half-space with respect to $F_\theta$, and the Heintze-Karcher-type inequality \eqref{ineq-hk-iso} can be rewritten as \footnote{For \eqref{ineq-HK-aniso}, we point out that a general Heintze-Karcher-type inequality in the half-space of this form has been proven in \cite{JWXZ23}, which holds for not only general convex gauges $F$, but also for a more general boundary contact angle condition.}
\eq{\label{ineq-HK-aniso}
\abs{\Om}
\leq\frac{n}{n+1}\int_\S \frac{F_\theta(\nu)}{H_{\S,F_\theta}}\rd\mcH^n,
}
where $H_{\S,F_\theta}$ is the anisotropic mean curvature of $\S$ with respect to $F_\theta$.

Let us use a few more words to explain this point, in fact, we will see in Lemma \ref{Lem-capillary-F-geometry} that the anisotropic outer normal, anisotropic mean curvature, and anisotropic principal eigenvalues of $\S$ with respect to $F_\theta$, are given by $\nu_{F_\theta}=\nu-\cos\theta E_{n+1}$, $H_{\S,F_\theta}=H_{\S}$, $\kappa_i^{F_\theta}=\kappa_i$, respectively.
In view of this, the shifted flow $\zeta_\theta$ can be then regarded as $\zeta_{F_\theta}$without shifting, in the sense that it is in fact an \textit{anisotropic geodesic normal flow}
\eq{
\zeta_{F_\theta}(x,t)
\coloneqq x-t\nu_F(x)
=x-t(\nu-\cos\theta E_{n+1})
=\zeta_\theta(x,t),
}
and hence \eqref{ineq-hk-iso-argument} reads as
\eq{\label{ineq-hk-aniso-argument}
&\abs{\Om}
\leq\abs{\zeta_{F_\theta}(Z)}
\leq\int_Z{\rm J}^Z\zeta_{F_\theta}\rd\mcH^{n+1}\\
=&\int_\S\int_0^{\frac1{\max_i\kappa_i^{F_\theta}(x)}}F_\theta(\nu)\prod_{i=1}^n(1-t\kappa^{F_\theta}_i(x))\rd\mcH^{n+1}
\leq\frac{n}{n+1}\int_\S\frac{F_\theta(\nu)}{H_{\S,F_\theta}}\rd\mcH^n.
}
%--------
\subsection{Outline of the proof}

As said, our proof is enlightened by the quantitative analysis of the Montiel-Ros argument conducted by Julin-Niinikoski \cite{JN23}, with significant modifications for the quantitative analysis to work in the capillary case.

Our first attempt in this direction is a Michael-Simon-type inequality for capillary hypersurfaces in the half-space, Theorem \ref{Thm-Michal-Simon}, which eliminates the boundary term $\int_{\p\S}f\rd\mcH^{n-1}$ that appears in the classical {Michael-Simon inequality} \cites{Allard72,MS73,Brendle21} by virtue of the capillary boundary condition.
As a direct consequence, we obtain a Topping-type inequality for capillary hypersurfaces in the half-space, Theorem \ref{Thm-Topping-ineq}.

Section \ref{Sec-4} is devoted to the a priori estimates resulting from the capillary structure.
A crucial and novel step is to use reflection and a new capillary gauge for the reflected surface inspired by \cite{JWXZ24},
which yields a new first variation structure in Lemma \ref{Lem:F-divergence-1st-variation}, by virtue of which
we obtain in Proposition \ref{Prop-JN23-Lem2.4} a priori estimates for $\lambda\in\mfR_+$\textemdash{the} prescribed constant that $H_\S$ is expected to be close to.
Moreover, we present a density-type estimate for the capillary hypersurfaces in Proposition \ref{Prop-density}, with the help of the Michael-Simon-type inequality.

The proof of our main result Theorem \ref{Thm-Stability} is intensively discussed in Section \ref{Sec-5}.
The core of our analysis is based on revisiting the Montiel-Ros-type argument \eqref{ineq-hk-aniso-argument} in a quantitative way.
Roughly speaking, we will have a close look at the error terms that arise each time we estimate with an inequality in \eqref{ineq-hk-aniso-argument}, and our goal is to show that these error terms are almost negligible, provided that the $L^n$-deficit is small enough.
This quantitative argument leads to 
Proposition \ref{Prop-JN23-Prop3.3}, in which the volume estimate \eqref{ineq-Om_r} for the super level-set with respect to the \textit{shifted distance function} to $\S$ (the capillary counterpart of the Euclidean distance function) is obtained.
With all these estimates, we are able to prove our main result.

The first step to approach Theorem \ref{Thm-Stability} is to find the points $o_1,\ldots,o_N$ that properly serve as the centers of our target caps and spheres.
Enlightened by the toy model case, it is natural to look at the super level-set $\Om_{r_0}$ (with respect to the shifted distance function) for $r_0$ that is close to the reference radius $R=\frac{n}\lambda$.
Once the centers are fixed, the rest of the proof will then be focused on proving that $\S$ is close to the union of spherical caps (spheres) given by $\bigcup_{i=1}^n\{\overline{\p B_{R;\theta}(o_i)\cap\mfR^{n+1}_+}\}$.

One obstacle that arises in the proof of Theorem \ref{Thm-Stability}, compared to the closed hypersurfaces case \cite{JN23}, is to confirm that these spherical caps (spheres) are actually close to $\theta$-caps or spheres that lie completely in $\overline{\mfR^{n+1}_+}$.
This will be discussed in {\bf Step 2}, and
the idea to solve this concern is based on the following observation:

If $\overline{\p B_{R;\theta}(o)\cap\mfR^{n+1}_+}$ is exactly a $\theta$-cap, that is, $o\in\p\mfR^{n+1}_+$, then $\overline{\p B_{r;\theta}(o)\cap\mfR^{n+1}_+}$ has to be a $\theta$-cap as well, for any $0<r<\infty$; if $\overline{\p B_{R;\theta}(o)\cap\mfR^{n+1}_+}$ is a sphere that is completely contained in the half-space, then so is $\overline{\p B_{r;\theta}(o)\cap\mfR^{n+1}_+}$ for any $0<r<R$.
In other words, if we look at the rescaled volume $\frac{\abs{B_{r;\theta}(o)\cap\mfR^{n+1}_+}}{r^{n+1}}$ of $\theta$-balls and complete balls, say $\mathfrak{b}_r$, then ${\rm osc}(\mathfrak{b}_r)$ is essentially vanishing, vice versa.
In view of this, we will first show in \eqref{ineq-b_theta-b_theta'} that for a small radius and a large enough radius which is close to $R$, ${\rm osc}(\mathfrak{b}_r)$ is controlled by the $L^n$-deficit.
Then by analyzing the expression of $\mathfrak{b}_r$, \eqref{eq-b_theta}, we obtain height estimates \eqref{ineq-o_i,E_n+1-Thm-1}, \eqref{ineq-o_i,E_n+1-Thm-2} of the centers $\{o_i\}$ from small oscillation, which solves the concern.

Another obstacle will be tackled in {\bf Step 3}, in which we try to show that the obtained balls are almost mutually disjoint.
In the closed hypersurfaces case, one needs only to show that the lower bound of $\abs{o_i-o_j}$ is almost $2R$, which is exactly what we wish to prove for the capillary case when $\theta\in[\frac\pi2,\pi)$, \eqref{ineq-o_i-o_j-1}.
For $\theta\in(0,\frac\pi2)$, we no longer expect this to be true, as one can easily see from the toy model case, where the boundaries of two $\theta$-caps of the same radii $R$ are almost mutually intersecting.
The Euclidean distance of their centers is then almost $2R\sin\theta$, not $2R$.
Fortunately, if we consider the anisotropic distance $F^o_\theta(o_i-o_j)$ of them, then $2R$ is again what we would expect.

We emphasize that, the second step is also crucial in the free boundary case.
Intuitively speaking, for free boundary hypersurfaces with small $L^n$-deficit, one may easily reflect them across the supporting hyperplane $\p\mfR^{n+1}_+$ to obtain $C^2$-closed hypersurfaces.
Since the reflection is isometric, the resulting closed hypersurfaces are of small $L^n$-deficit as well, and hence one may exploit \cite{JN23} to conclude that such closed hypersurfaces are close in the sense of Hausdorff distance to a family of disjoint spheres with equal radii.
However, this does not automatically imply that $\S$ itself has to be close to a family of disjoint free boundary caps.
More evidence is needed towards the ultimate goal, for example, the height estimates \eqref{ineq-o_i,E_n+1-Thm-1}, \eqref{ineq-o_i,E_n+1-Thm-2}.

%========
\subsection{Acknowledgements}
The authors would like to thank Prof. Chao Xia for many insightful discussions and constant support.
X.Z. would also like to thank Prof. Julian Scheuer for his constant support and for many valuable discussions
on this project.
X.J. is supported by National Natural Science Foundation of China
(Grant No. 12401249), Natural Science Foundation of Jiangsu Province, China (Grant No. BK20241258).

%========
\section{Notations and preliminaries}\label{Sec-2}
%-----------
\subsection{Notations}

We will be working on the Euclidean space $\mbR^{n+1}$ for $n\geq1$, with  the Euclidean scalar product denoted by $\left<\cdot,\cdot\right>$ and the corresponding Levi-Civita connection denoted by $\na$. $\mbR^{n+1}_+=\{x:x_{n+1}>0\}$ is the open upper half-space and $E_{n+1}\coloneqq(0,\ldots,0,1)$.
We also write $\{E_1,\ldots,E_{n+1}\}$ to denote the coordinates basis for $\mfR^{n+1}$.
For $k\in\mbN^+$, $\mcH^k$ denotes the $k$-dimensional Hausdorff measure, and $\mcL^{n+1}$ denotes the Lebesgue measure in $\mfR^{n+1}$.
The Minkowski sum of 
two sets in $\mfR^{n+1}$ is denoted as
\eq{
A+B=\{x+y\in\mfR^{n+1}:x\in A, y\in B\}.
}
We adopt the following conventions regarding the use of the symbol $\abs{\cdot}$.
If we plug in an open set $\Om$ of $\mfR^{n+1}$, then we write
\begin{align*}
\abs{\Om}
\coloneqq\mathcal{L}^{n+1}(\Om).
\end{align*}
If we plug in a $k$-dimensional submanifold $M\subset\mfR^{n+1}$, then we mean
\begin{align*}
\abs{M}
\coloneqq\mathcal{H}^k(M).
\end{align*}
If we plug in a vector $v\in\mfR^{n+1},$ then $\abs{v}$ denotes the Euclidean length of $v$.

Let $\Om$ be a bounded open set (possibly not connected) in $\mbR^{n+1}_+$ with piecewise $C^2$-boundary $\p\Om=\S\cup T$, where $\S=\overline{\p\Om\cap\mbR^{n+1}_+}$ is a $\theta$-capillary hypersurface in $\overline{\mbR^{n+1}_+}$ and $T=\p\Om\cap\p\mbR^{n+1}_+$. 
In this paper, we will always assume that $\abs{\S}>0$, $\abs{T}>0$,
and that the corner given by $\Gamma\coloneqq\S\cap T=\p\S=\p T$, is a smooth codimension two submanifold in $\mfR^{n+1}$ with $\abs{\Gamma}>0$.

%-----------

\subsection{Capillary hypersurfaces in the half-space}\label{Sec-2.1}

We use the following notation for normal vector fields. Let $\nu$ and $\bar N=-E_{n+1}$ be the outward unit normal to $\S$ and $\p\mbR^{n+1}_+$ (with respect to $\Om$) respectively. Let $\mu$ be the outward unit co-normal to $\Gamma=\p\S\subset \S$ and $\bar \nu$ be the outward unit co-normal to $\Gamma=\p T\subset T$. Under this convention, along $\Gamma$ the bases $\{\nu,\mu\}$ and $\{\bar \nu,\bar N\}$ have the same orientation in the normal bundle of $\p \S\subset \mbR^{n+1}$.
In particular, $\S$ is $\theta$-capillary if along $\Gamma$, \eq{\label{condi-mu-nu}
&&\mu=\sin \theta \bar N+\cos\theta \bar \nu,\quad \nu=-\cos \theta \bar N+\sin \theta \bar \nu.
}

We denote by $\na$,
%$\Delta$, $\na^2$ and
${\rm div}$, the gradient,
%the Laplacian, the Hessian
and the divergence operator on $\mfR^{n+1}$ respectively, while by $\na^\S$, ${\rm div}_\S$ the gradient, and the divergence on $\S$, respectively.
Let $g$, $h$ and $H$ be the first, second fundamental forms and the mean curvature  of the smooth part of $\S$ respectively. Precisely, $h(X, Y)=\<\na_X\nu, Y\>$ and $H={\rm tr}_g(h)$.
Finally, we use ${\rm dist}_g(\cdot,\cdot)$ to denote the distance on $\S$ that is induced from $g$.

\subsubsection{Capillary gauge}\label{Sec-2-1-1}
Given a prescribed capillary angle $\theta\in(0,\pi)$, we consider the \textit{capillary gauge} $F_\theta:\mfR^{n+1}\ra\mfR^+$, defined as \eqref{defn-F_theta},
which is a \textit{convex gauge} in the sense that: when restricted to $\mbS^n$, $F_\theta$ is a smooth positive function, with
\eq{
A_{F_\theta}
\coloneqq D^2 F_\theta+F_\theta\sigma
}
positive definite, where $\sigma$ is the canonical metric on $\mbS^n$ and $D$ is the corresponding Levi-Civita connection on $\mbS^n$.
The \textit{Cahn-Hoffman map} associated with $F_\theta$ is given by
\eq{
\Phi:\mbS^n\to\mfR^{n+1},\text{ } \Phi(z)
=DF_\theta(z)+F_\theta(z)z,
}

We shall suppress the dependence of $F_\theta$ on $\theta$ and denote it simply by $F$ in all follows.
Moreover, the \textit{dual gauge} of $F$ is denoted by $F^o$, which is given by
\eq{
F^o(x)
=\sup\left\{\frac{\left<x,z\right>}{F(z)}: z\in\mbS^n\right\},
}
where $\left<\cdot,\cdot\right>$
denotes the standard Euclidean scalar product.

The following are well-known facts regarding convex gauge, see e.g., \cite{JWXZ23}. 

\begin{proposition}\label{basic-F}
For any $z\in \mbS^n$ and  $t>0$, the following statements hold.
\begin{itemize}
    \item[(i)] $F^o(tz)=tF^o(z)$.
\item[(ii)] $\left<\Phi(z),z\right>=F(z)$.
\item[(iii)] $F^o(\Phi(z))=1$.
%\item[(iv)] The following Cauchy-Schwarz inequality holds:
%\begin{align}\label{Cau-Sch}\<x, z\>\le F^o(x)F(z).
%\end{align}
\item[(iv)]  The unit
Wulff shape $\p\mcW$ can be interpreted by $F^o$ as \eq{
\p\mcW=\{x\in\mbR^{n+1}|F^o(x)=1\}.
}
\end{itemize}
\end{proposition}
A Wulff shape of radius $r$ centered at $x_0\in \mfR^{n+1}$ is given by \begin{align*}
\p\mcW_{r_0}(x_0)
=\{x\in\mbR^{n+1}|F^o(x-x_0)=r_0\}.
\end{align*}

%----------
We collect some background materials from geometric measure theory, and we refer to the monograph \cite{Mag12} for a detailed account.
	
Let $\Om$ be a Lebesgue measurable set in $\mfR^{n+1}$, we say that $\Om$ is a \textit{set of finite perimeter in $\mfR^{n+1}$} if
\eq{
\sup\left\{\int_\Om
{\rm div}X\rd\mcL^{n+1}: X\in C^1_c(\mfR^n;\mfR^n),  \abs{X}\leq1\right\}<\infty.
}
An equivalent characterization of sets of finite perimeter is that: there exists a $\mfR^{n+1}$-valued Radon measure $\mu_\Om$ on $\mfR^{n+1}$ such that for any $X\in C_c^1(\mfR^{n+1};\mfR^{n+1})$,
\eq{
\int_\Om{\rm div}X\rd\mcL^{n+1}
=\int_{\mfR^{n+1}}\left<X,\rd\mu_\Om\right>.
}
$\mu_\Om$ is called the \textit{Gauß-Green measure} of $\Om$.
The \textit{relative perimeter} of $\Om$ in an open subset $A\subset\mfR^{n+1}$, and the \textit{perimeter} of $\Om$, are defined as
\eq{
P(\Om;A)
=\Abs{\mu_\Om}(A),\text{ }
P(\Om)
=\Abs{\mu_\Om}(\mfR^{n+1}).
}
Given a convex gauge $F$,
the \textit{anisotropic perimeter relative to $\mfR^{n+1}_+$} is defined by
\eq{
P_F(\Om;\mfR^{n+1}_+)
=\sup\left\{\int_{\Om\cap\mfR^{n+1}_+}{\rm div}X\rd\mcL^{n+1}: X\in C_0^1(\mfR^{n+1}_+;\mfR^{n+1}), F^o(X)\leq1\right\}.
}
One can check by definition that the quantity $P_F(\Om;\mfR^{n+1}_+)$ is finite if and only if the classical relative perimeter $P(\Om;\mfR^{n+1}_+)<\infty$.
In particular, for a set of finite perimeter $\Om\subset\mfR^{n+1}_+$, the anisotropic perimeter relative to $\mfR^{n+1}_+$ (anisotropic surface energy) can be characterized by
\eq{\label{eq-P_F}
P_F(\Om;\mfR^{n+1}_+)
=\int_{\p^\ast \Om\cap \mfR^{n+1}_+}F(\nu_\Om)\rd\mcH^n,
}
where $\p^\ast\Om$ is the \textit{reduced boundary} of $\Om$ and $\nu_\Om$ is the \textit{measure-theoretic outer unit normal} to $\Om$.
Note that if $\Om$ is of $C^1$-boundary
in $\mfR^{n+1}_+$, then $\nu_\Om$ agrees with the classical outer unit normal $\nu$.

%----------

We record the following facts resulting from the definition of capillary gauge.
\begin{lemma}\label{Lem-capillary-F}
For the capillary gauge $F$ and for any $x,\xi\in\mfR^{n+1}$, there hold
\begin{enumerate}
%------
\item \cite{LXZ23}*{(3.1)}:
\eq{
\na F(\xi)
=\frac{\xi}{\abs{\xi}}-\cos\theta E_{n+1},
}
%----
\item \cite{LXZ23}*{Propositions 3.1, 3.2}:
\eq{
F^o(x)
=\frac{\abs{x}^2}{\sqrt{\cos^2\theta\left<x, E_{n+1}\right>^2+\sin^2\theta\abs{x}^2}-\cos\theta x\cdot E_{n+1}},
}
and the unit Wulff shape with respect to $F$ is given by
\eq{
\p\mcW_F
=&\{x:F^o(x)=1\}\\
=&\na F(\mbS^n)
=\mbS^n-\cos\theta E_{n+1}
=\{x:\abs{x+\cos\theta E_{n+1}}=1\}.
}
%------
\item \cite{LXZ23}*{Proposition 3.3}:

For any set of finite perimeter $\Om\subset\mfR^{n+1}_+$, the relative anisotropic perimeter with respect to $F$ is given by
\eq{
P_F(\Om;\mfR^{n+1}_+)
=P(\Om;\mfR^{n+1}_+)-\cos\theta P(\Om;\p\mfR^{n+1}_+).
}
\end{enumerate}
\end{lemma}
It is then easy to see that the open unit Wulff ball is exactly given by
\eq{
\mcW_F
=B_1(-\cos\theta E_{n+1})
\eqqcolon B_{1;\theta}.
}
For simplicity, we adopt the following conventions: for any (anisotropic) radius $\rho>0$, the (Wulff) ball centered at the origin with radius $\rho$ is denoted by $B_\rho$ ($\mcW_\rho=B_{\rho;\theta}$).

If we denote the minimum and the maximum values of $F$ on $\mbS^n$ by
\eq{
m_F=\min_{\mbS^n}F,\quad
M_F=\max_{\mbS^n}F,
}
then we have
\eq{\label{defn-m-M-F}
m_F=1-\abs{\cos\theta},\quad
M_F=1+\abs{\cos\theta}.
}
Consequently, for the dual gauge $F^o$, one has
\eq{\label{defn-m-M-F^o}
m_{F^o}
=\frac1{M_F}=\frac1{1+\abs{\cos\theta}},\quad
M_{F^o}
=\frac1{m_F}=\frac1{1-\abs{\cos\theta}}.
}

\subsubsection{Capillary geometry and anisotropic geometry}
Let $\S\subset\overline{\mfR^{n+1}_+}$ be a $C^2$-hypersurface with $\p\S\subset \p\mfR^{n+1}_+$, which encloses a bounded domain $\Om$.
Let $\nu$ be the unit normal of $\S$ pointing outward $\Om$. 
The \textit{anisotropic normal} of $\S$ is given
\eq{
\nu_F
=\Phi(\nu)
=\na F(\nu)
=D F(\nu)+F(\nu)\nu,
}
and the \textit{anisotropic principal curvatures} $\{\kappa_i^F\}_{i=1}^n$ of $\S$ are given by the eigenvalues of the \textit{anisotropic Weingarten map} \eq{
\rd\nu_F
=A_F(\nu)\circ \rd\nu: T_x\S\to T_x\S. 
}
%The eigenvalues are real since $(A_F)$ is positive definite and symmetric.
Correspondingly, the \textit{anisotropic mean curvature} of $\S$ is given by $H_F=\sum_{i=1}^n\kappa_i^F$.

Invoking the definition of capillary gauge, we may use direct computations to see that: 
\begin{lemma}\label{Lem-capillary-F-geometry}
Given $\theta\in(0,\pi)$, let $F=F_\theta$ be the capillary gauge and let $\S\subset\overline{\mfR^{n+1}_+}$ be a $C^2$-hypersurface.
Then at any $x\in\S$,
\begin{enumerate}
\item $F(\nu(x))=1-\cos\theta\left<\nu(x),E_{n+1}\right>$;
\item $\nu_F(x)=\nu(x)-\cos\theta E_{n+1}$;
\item $\rd\nu_F\mid_x=\rd\nu\mid_x$, that is, the anisotropic Weingarten map is in fact the classical Weingarten map.
Consequently, we have: 
$\kappa_i^F(x)=\kappa_i(x)$, and of course
$H_{\S,F}(x)=H_\S(x)$. 
\end{enumerate}
\end{lemma}

%===========
\subsection{More on capillarity}
\begin{proposition}[Trace estimate]\label{Prop-IBP-Capillary}
Given $\theta\in(0,\pi)$, let $\S\subset\overline{\mbR^{n+1}_+}$ be a compact $\theta$-capillary hypersurface, then it holds that: for any smooth function $f$ on $\S$,
\eq{\label{eq-divf-capillary}
-\sin\theta\int_{\p\S}f\rd\mcH^{n-1}
=\int_\S\left<\na^\S f,E_{n+1}\right>\rd\mcH^n-\int_\S Hf\left<\nu,E_{n+1}\right>\rd\mcH^n.
}
In particular,
\eq{\label{eq-divf-capillary-2}
\sin\theta\abs{\p\S}
=\int_\S H\left<\nu,E_{n+1}\right>\rd\mcH^n.
}
Moreover, one has
\eq{\label{ineq-JXZ23-(20)}
\int_\S\left<\nu,E_{n+1}\right>\rd\mcH^n
=\abs{T},
}
so that
\eq{\label{ineq-T<=Sigma}
\abs{T}\leq\abs{\S}.
}
\end{proposition}
\begin{proof}
For any $x\in\S$, we denote by $E_{n+1}^T(x)$ the tangential part of $E_{n+1}$ with respect to $T_x\S$, a direct computation then yields
\eq{
{\rm div}_\S(fE_{n+1}^T)
=\left<\na^\S f,E_{n+1}\right>-Hf\left<\nu,E_{n+1}\right>,
}
thus
\eq{
\int_{\p\S}f\left<\mu,E_{n+1}\right>\rd\mcH^{n-1}
=&\int_\S\div_\S(fE_{n+1}^T)\rd\mcH^n\\
=&\int_\S\left<\na^\S f,E_{n+1}\right>\rd\mcH^n-\int_\S Hf\left<\nu,E_{n+1}\right>\rd\mcH^n,
}
\eqref{eq-divf-capillary} follows from the capillary condition \eqref{condi-mu-nu}.
Choosing $f=1$ in \eqref{eq-divf-capillary}, we obtain \eqref{eq-divf-capillary-2}.

\eqref{ineq-JXZ23-(20)} can be found in \cite{JXZ23}*{(20)}, and the last assertion follows easily from \eqref{ineq-JXZ23-(20)}, which completes the proof.
\end{proof}

For a fixed unit sphere, we know that either it has a fixed contact angle $\theta\in(0,\pi)$ with $\p\mfR^{n+1}_+$, or it stays away from (at most mutually tangent with) $\p\mfR^{n+1}_+$.
In the latter case, we adopt the convention that $\theta=\pi$.
In both cases, we use 
\eq{
\mathfrak{b}_\theta
}
to denote the volume of the enclosed region of it with the supporting hyperplane $\p\mfR^{n+1}_+$.
Note that $\mathfrak{b}_\theta$ can be explicitly computed, see e.g., \cite{Li11}:
\eq{\label{eq-b_theta}
\mathfrak{b}_\theta
%=&\frac{\om_n}2I_{\sin^2\theta}(\frac{n}2,\frac12)-\frac{\om_{n-1}}n\cos\theta\sin^n\theta\\
=
\begin{cases}
\frac{\om_{n+1}}2I_{\sin^2\theta}(\frac{n+2}2,\frac12),\quad&\theta\in(0,\frac\pi2),\\
\om_{n+1}-\frac{\om_{n+1}}2I_{\sin^2\theta}(\frac{n+2}2,\frac12),\quad&\theta\in[\frac\pi2,\pi),
\end{cases}
}
where $\om_{n+1}$ is the Lebesgue measure of unit ball in $\mfR^{n+1}$, $I_{\sin^2\theta}(\frac{n+2}2,\frac12)$ is the \textit{regularized incomplete beta function} given by
\eq{
I_{\sin^2\theta}(\frac{n+2}2,\frac12)
=\frac{\int_0^{\sin^2\theta}t^{\frac{n}2}(1-t)^{-\frac12}\rd t}
{\int_0^1t^{\frac{n}2}(1-t)^{-\frac12}\rd t}.
%\eqqcolon\frac{B(s;\frac{k}2,\frac12)}{B(1;\frac{k}2,\frac12)}.
}
It follows that $\mathfrak{b}_\theta$ is increasing on $\theta\in(0,\pi]$.

We have the following monotonicity lemma.
\begin{lemma}\label{Lem-monotonicity-theta}
Given $o\in\overline{\mfR^{n+1}_+}$, $\theta\in(0,\pi)$.
For any $0<\rho<\infty$, let $\theta_\rho\in(0,\pi]$ be the contact angle of $\p B_{\rho;\theta}(o)$ with $\p\mfR^{n+1}_+$.
Then for any $0<\rho_1<\rho_2<\infty$, there hold
\eq{
\theta\leq\theta_{\rho_2}\leq\theta_{\rho_1}\leq\pi,\\
\mathfrak{b}_\theta\leq\mathfrak{b}_{\theta_{\rho_2}}\leq\mathfrak{b}_{\theta_{\rho_1}}\leq\om_{n+1}.
}
Moreover, the equality case $\theta=\theta_{\rho_2}$ happens if and only if $o\in\p\mfR^{n+1}_+$.
\end{lemma}
\begin{proof}
Notice that if $B_{\rho;\theta}(o)\cap\p\mfR^{n+1}_+=\emptyset$, then we readily have
\eq{
\theta_\rho=\pi,\text{ }\mathfrak{b}_{\theta_\rho}=\mathfrak{b}_{\pi}=\om_{n+1}.
}
In the case that $\p B_{\rho;\theta}(o)\cap\p\mfR^{n+1}_+\neq\emptyset$, let us fix any $z\in\p B_{\rho;\theta}(o)\cap\p\mfR^{n+1}_+$.

Clearly, one has
\eq{\label{ineq-theta_rho-theta}
\cos\theta_{\rho}
=\left<\frac{z-(o-\rho\cos\theta E_{n+1})}\rho,E_{n+1}\right>
=\cos\theta-\frac{\left<o,E_{n+1}\right>}\rho,
}
since $\left<o,E_{n+1}\right>\geq0$, it follows that $\cos\theta_\rho$ is non-decreasing and hence $\theta_\rho$ is non-increasing on $\rho$. Moreover, $\cos\theta_\rho\leq\cos\theta$, implying that $\theta_\rho\geq\theta$, with equality holds if and only if $\left<o,E_{n+1}\right>=0$.

The monotonicity of $\mathfrak{b}_{\theta_\rho}$ follows immediately, which completes the proof.
\end{proof}

%==========

\section{Michael-Simon-type inequality and Topping-type inequality}
\begin{theorem}\label{Thm-Michal-Simon}
Given $\theta\in(0,\pi)$,
let $\S\subset\overline{\mbR^{n+1}_+}$ ($n\in\mbN$) be a compact $\theta$-capillary hypersurface.
Let $f$ be a positive smooth function on $\S$.
If $n\geq2$, then 
\eq{\label{ineq-Michael-Simon}
\norm{f}_{L^\frac{n}{n-1}(\S)}
\leq\sigma(n,\theta)\left(\int_\S\abs{\na^\S f}\rd\mcH^n+\int_\S f\abs{H_\S}\rd\mcH^n\right)
}
for some positive constant $\sigma$ that depends on $n,\theta$.

If $n=1$, then
\eq{\label{ineq-Michael-Simon-n=1}
\int_{\S\cap\{x:f(x)\geq1\}}f\rd\mcH^1
\leq\sigma(n=1,\theta)\int_\S f\rd\mcH^1\left(\int_\S\abs{\na^\S f}\rd\mcH^1+\int_\S f\abs{H_\S}\rd\mcH^1\right)
}
for some positive constant $\sigma$ that depends on $n=1,\theta$.
\end{theorem}
\begin{proof}
{\bf Case 1. }$n\geq2$.

Define $\bar\sigma(n)=n\left(\frac{(n+2)\om_{n+1}}{2\om_2}\right)^\frac1n$, then we obtain from the Michael-Simon inequality \cite{Brendle21}*{Theorem 1} (with $m=2$ chosen therein) and the Cauchy-Schwarz inequality:
\eq{
\bar\sigma(n)\norm{f}_{L^\frac{n}{n-1}(\S)}
\leq\int_{\p\S}f\rd\mcH^{n-1}+\int_\S\abs{\na^\S f}\rd\mcH^n+\int_\S f\abs{H_\S}\rd\mcH^n.
}
Taking \eqref{eq-divf-capillary} into account, we obtain
\eq{
\bar\sigma(n)\norm{f}_{L^\frac{n}{n-1}(\S)}
\leq&-\frac{1}{\sin\theta}\int_\S\left<\na^\S f,E_{n+1}\right>\rd\mcH^n+\frac{1}{\sin\theta}\int_\S H_\S f\left<\nu,E_{n+1}\right>\rd\mcH^n\\
&+\int_\S\abs{\na^\S f}\rd\mcH^n+\int_\S f\abs{H_\S}\rd\mcH^n\\
\leq&(1+\frac{1}{\sin\theta})\left(\int_\S\abs{\na^\S f}\rd\mcH^n+\int_\S f\abs{H_\S}\rd\mcH^n\right),
}
from which we conclude \eqref{ineq-Michael-Simon}.

{
{\bf Case 2. }$n=1$.

In this case, we write $V=\mcH^1\llcorner\S\otimes\de_{T_x\S}$ as the $1$-varifold in $\mfR^2$ induced by the capillary hypersurface $\S$.
It is easy to check that $\de V=H_\S\nu\mcH^1\llcorner\S+\mu\mcH^0\llcorner\p\S$, and the total variation of first variation is just $\norm{\de V}=\abs{H_\S}\mcH^1\llcorner\S+\mcH^0\llcorner\p\S$, a Radon measure on $\mfR^2$.

We use \cite{Allard72}*{Theorem 7.1} (with $k=1$ therein) for the function $f$ on the induced $1$-varifold $V$, to obtain
\eq{
\int_{\S\cap\{x:f(x)\geq1\}}f\rd\mcH^1
\leq C\int_\S f\rd\mcH^1\left(\int_{\p\S}f\rd\mcH^0+\int_\S\abs{\na^\S f}\rd\mcH^1+\int_\S f\abs{H_\S}\rd\mcH^1\right)
}
for some positive constant $C$ depending only on $n=1$.
Taking \eqref{eq-divf-capillary} into account, the assertion follows.
}
\end{proof}

As a by-product, we obtain the following Topping-type inequality, which controls the upper bound of the extrinsic diameter of $\S$, denoted by
\eq{
d_{\rm ext}(\S)
\coloneqq\max_{x,y\in\S}\abs{x-y}.
}
\begin{theorem}\label{Thm-Topping-ineq}
Given $n\in\mbN^+$, $\theta\in(0,\pi)$,
let $\S\subset\overline{\mbR^{n+1}_+}$ be a compact, connected $\theta$-capillary hypersurface.
There holds that
\eq{\label{ineq-Topping-capillary}
d_{\rm ext}(\S)
\leq C(n,\theta)\int_\S\abs{H}^{n-1}\rd\mcH^n,
}
for some positive constant $C$ depending only on $n,\theta$.
\end{theorem}

The proof follows essentially from \cite{Topping08} and we present it in Appendix \ref{App-1} for readers' convenience.

%=========
\section{A priori estimates}\label{Sec-4}

{
In this section, we establish the a priori estimates needed in the proof of our main theorem.
%-------
\subsection{Reflection and anisotropy}
To our aim, we consider the reflection across the supporting hyperplane as follows:
Let $\Om$ be a bounded open set (possibly not connected) in $\mbR^{n+1}_+$ with piecewise $C^2$-boundary $\p\Om=\S\cup T$, where $\S=\overline{\p\Om\cap\mbR^{n+1}_+}$ is a $\theta$-capillary hypersurface in $\overline{\mbR^{n+1}_+}$ and $T=\p\Om\cap\p\mbR^{n+1}_+$.
Let $\mathscr{R}:\mfR^{n+1}\ra\mfR^{n+1}, x\mapsto x-2x_{n+1}E_{n+1}$ be the reflection across $\p\mfR^{n+1}$, and denote by
\eq{
\widetilde\Om
=\mathscr{R}(\Om),\quad
\widetilde\S
=\mathscr{R}(\S)
}
the reflected set and hypersurface respectively, then write the unions as
\eq{
\widehat\Om
=\Om\cup\widetilde\Om,\quad
\widehat\S
=\S\cup\widetilde\S.
}
Clearly, $\widehat\Om$ has piecewise $C^2$-boundary $\p\widehat\Om=\widehat\S$ and hence is a set of finite perimeter.
Also,
\eq{\label{eq:widehatOm-widehatSigma}
\abs{\widehat\Om}
=2\abs{\Om},\quad
\abs{\widehat\S}
=2\abs{\S}.
}

\begin{lemma}\label{Lem:1stvariation-widehat-S}
For any $X\in C_c^1(\mfR^{n+1},\mfR^{n+1})$, the first variation formula holds:
\eq{
\int_{\widehat\S}{\rm div}_{\widehat\S}X\rd\mcH^n
-2\cos\theta\int_T{\rm div}_{\p\mfR^{n+1}_+}X\rd\mcH^n
=\int_{\widehat\S}H\left<X,\nu\right>\rd\mcH^n,
}
where $H=H_\S$ on $\S$ and $H=H_{\widetilde\S}$ on $\widetilde\S$, $\nu$ is the outer unit co-normal with respect to $\widehat\Om$ along $\widehat\S$ whenever it is well-defined.
\end{lemma}
\begin{proof}
Since $\S$ is a $\theta$-capillary hypersurface, the outer unit co-normal $\mu=\cos\theta\bar\nu-\sin\theta E_{n+1}$ along $\p\S$.
By using tangential divergence theorem on $\S$ and $T$ (recall that $T$ is flat), we find
\eq{\label{formu-1st-variation}
\int_\S{\rm div}_\S X\rd\mcH^n
=&\int_\S H_\S\left<X,\nu\right>\rd\mcH^n
+\int_{\p\S}\left<X,\mu\right>\rd\mcH^{n-1}\\
=&\int_\S H_\S\left<X,\nu\right>\rd\mcH^n+\cos\theta\int_{\p\S}\left<X,\bar\nu\right>\rd\mcH^{n-1}-\sin\theta\int_{\p\S}\left<X,E_{n+1}\right>\rd\mcH^{n-1}\\
=&\int_\S H_\S\left<X,\nu\right>\rd\mcH^n+\cos\theta\int_{T}{\rm div}_{\p\mfR^{n+1}_+}X\rd\mcH^n-\sin\theta\int_{\p\S}\left<X,E_{n+1}\right>\rd\mcH^{n-1}.
}
Note that $\widetilde\S$ is also a $\theta$-capillary hypersurface, with the outer unit co-normal $\mu_{\widetilde\S}=\cos\theta\bar\nu+\sin\theta E_{n+1}$,
repeating the above computation we obtain
\eq{
\int_{\widetilde\S}{\rm div}_{\widetilde\S}X\rd\mcH^n=\int_{\widetilde\S}H_{\widetilde\S}\left<X,\nu\right>\rd\mcH^n+\cos\theta\int_{T}{\rm div}_{\p\mfR^{n+1}_+}X\rd\mcH^n+\sin\theta\int_{\p\S}\left<X,E_{n+1}\right>\rd\mcH^{n-1}.
}
Summing these two identities, the assertion follows.
\end{proof}

To proceed, motivated by \cite{JWXZ24}, we consider the gauge function
\eq{
F_\ast(\xi)
\coloneqq F(-\xi),\quad
\xi\in\mfR^{n+1},
}
which is induced geometrically by the convex body that is centrally
symmetric to $\p\mcW$.
In our case, it is easy to see that $F_\ast(\xi)=\abs{\xi}+\cos\theta\left<\xi,E_{n+1}\right>$, and hence
\eq{\label{defn-F-F_ast}
F=F_\theta,\quad
F_\ast=F_{\pi-\theta}.
}
\begin{lemma}\label{Lem:S-F-F_ast}
Let $\widehat\S=\S\cup\widetilde\S$ and the gauge functions $F$, $F_\ast$ be given as above.
For any $X\in C_c^1(\mfR^{n+1},\mfR^{n+1})$,
let $\{f_t\}_{\abs{t}<\ep}$ be the induced one parameter family of $C^1$-diffeomorphisms of $X$, i.e., $\{f_t\}_{\abs{t}<\ep}$ solves the Cauchy's problems
\eq{
\frac{\p}{\p t}f_t(\xi)=X(f_t(\xi)),\quad\xi\in\mfR^{n+1},\\
f_0(\xi)
=\xi,\quad\xi\in\mfR^{n+1}.
}
Then there holds
\eq{
\frac{\rd}{\rd t}\mid_{t=0}\left(\int_{f_t(\S)} F(\nu_t)\rd\mcH^n
+\int_{f_t(\widetilde\S)}F_\ast(\nu_t)\rd\mcH^n\right)
%=\abs{\widehat\S}-2\cos\theta\abs{T}.
=\int_{\widehat\S}{\rm div}_{\widehat\S}X\rd\mcH^n
-2\cos\theta\int_T{\rm div}_{\p\mfR^{n+1}_+}X\rd\mcH^n.
}
\end{lemma}
\begin{proof}
Since $f_t$ is of class $C^1$, we know that $f_t(\widehat\Om)=f_t(\Om)\cup f_t(\widetilde\Om)$ has piecewise regular boundary, and the sets $f_t(\Om)$, $f_t(\widetilde\Om)$ share the common boundary $f_t(T)$.
Denote by $\bar N_t$ the outer unit normal along $f_t(T)$ with respect to $f_t(\Om)$, then the outer unit normal along $f_t(T)$ with respect to $f_t(\widetilde\Om)$ is just $-\bar N_t$.
Denote by $\nu_t$ the outer unit normal along $f_t(\widehat\S)=f_t(\S)\cup f_t(\widetilde\S)$ with respect to $f_t(\widehat\Om)$ whenever it is well-defined.

Since ${\rm div}(\pm E_{n+1})=0$, using divergence theorem, we find
\eq{
0
=&\int_{f_t(\Om)}{\rm div}E_{n+1}\rd x
=\int_{f_t(\S)}\left<\nu_t,E_{n+1}\right>\rd\mcH^n+\int_{f_t(T)}\left<\bar N_t,E_{n+1}\right>\rd\mcH^n,\\
0=&\int_{f_t(\widetilde\Om)}{\rm div}(-E_{n+1})\rd x
=\int_{f_t(\widetilde\S)}\left<\nu_t,-E_{n+1}\right>\rd\mcH^n
+\int_{f_t(T)}\left<\bar N_t,E_{n+1}\right>\rd\mcH^n,
}
taking \eqref{defn-F-F_ast} into account, we have
\eq{\label{eq:f_t(S)}
\int_{f_t(\S)}F(\nu_t)\rd\mcH^n
=&\abs{f_t(\S)}-\cos\theta\int_{f_t(\S)}\left<\nu_t,E_{n+1}\right>\rd\mcH^n\\
=&\abs{f_t(\S)}+\cos\theta\int_{f_t(T)}\left<\bar N_t,E_{n+1}\right>\rd\mcH^n,
}
and
\eq{\label{eq:f_t(tilde-S)}
\int_{f_t(\widetilde\S)}F_\ast(\nu_t)\rd\mcH^n
=&\abs{f_t(\widetilde\S)}+\cos\theta\int_{f_t(\widetilde\S)}\left<\nu_t,E_{n+1}\right>\rd\mcH^n\\
=&\abs{f_t(\widetilde\S)}+\cos\theta\int_{f_t(T)}\left<\bar N_t,E_{n+1}\right>\rd\mcH^n.
}
Note also that by area formula, we have
\eq{
\abs{f_t(\widehat\S)}
=&\abs{f_t(\S)}+\abs{f_t(\widetilde\S)}
=\int_\S{\rm J}^\S f_t\rd\mcH^n+\int_{\widetilde\S}{\rm J}^{\widetilde\S}f_t\rd\mcH^n,\\
\int_{f_t(T)}\left<\bar N_t,E_{n+1}\right>\rd\mcH^n
=&\int_T\left<\bar N_t(f_t(x)),E_{n+1}\right>{\rm J}^Tf_t(x)\rd\mcH^n(x).
%=-\int_T{\rm J}^Tf_t\rd\mcH^n,
}
%where we have used the fact that $\bar N_0=-E_{n+1}$.
Since $\{E_1,\ldots,E_n\}$ is an orthonormal frame of $T\subset\p\mfR^{n+1}_+$, by $\left<\p_t\bar N_t(f_t(x)),\bar N_t(f_t(x)\right>=0$, we get
\eq{
&\p_t\bar N_t(f_t(x))
=\sum_{i=1}^n\left<\p_t\bar N_t(f_t(x)),\na_{E_i}(f_t(x))\right>\na_{E_i}(f_t(x))\\
=&-\sum_{i=1}^n\left<\bar N_t(f_t(x)),\p_t\na_{E_i}(f_t(x))\right>\na_{E_i}(f_t(x))\\
=&-\sum_{i=1}^n\left<\bar N_t(f_t(x)),\na_{E_i}\p_t(f_t(x))\right>\na_{E_i}(f_t(x))
=-\sum_{i=1}^n\left<\bar N_t(f_t(x)),\na_{E_i}X(f_t(x))\right>\na_{E_i}(f_t(x)),
}
so that by virtue of $f_0={\rm id}$, we obtain
\eq{
\frac{\rd}{\rd t}\mid_{t=0}\left<\bar N_t(f_t(x)),E_{n+1}\right>
=-\sum_{i=1}^n\left<\bar N_0(x),\na_{E_i}X(x)\right>\left<E_i,E_{n+1}\right>
=0.
}
Taking into account that
$\frac{\rd}{\rd t}\mid_{t=0}{\rm J}^Tf_t(x)={\rm div}_{\p\mfR^{n+1}_+}X(x)$, and $\bar N_0=-E_{n+1}$, we thus find
\eq{
\frac{\rd}{\rd t}\mid_{t=0}
\left(\int_T\left<\bar N_t(f_t(x)),E_{n+1}\right>{\rm J}^Tf_t(x)\rd\mcH^n(x)\right)
%=&0+\int_T\left<\bar N_0(x),E_{n+1}\right>{\rm div}_TX(x)\rd\mcH^n(x)\\
=-\int_T{\rm div}_{\p\mfR^{n+1}_+}X\rd\mcH^n.
}
Summing \eqref{eq:f_t(S)} and \eqref{eq:f_t(tilde-S)}, then differentiating the equation, the assertion follows. 
\end{proof}

\begin{lemma}\label{Lem:F-divergence-1st-variation}
Under the notations above, the first variation formula holds:
For any $X\in C_c^1(\mfR^{n+1},\mfR^{n+1})$,
\eq{
\int_\S F(\nu){\rm div}_{\S,F}X\rd\mcH^n
+\int_{\widetilde\S} F_\ast(\nu){\rm div}_{\widetilde\S,F_\ast}X\rd\mcH^n
=\int_{\widehat\S}H\left<X,\nu\right>\rd\mcH^n,
}
where ${\rm div}_{\S,F}X\coloneqq{\rm div}X-\frac{\left<\na F(\nu),(\na X)^\ast[\nu]\right>}{F(\nu)}$ is called the boundary $F$-divergence of $X$ with respect to $\S$, and
${\rm div}_{\widetilde\S,F_\ast}X$ is understood similarly.
\end{lemma}
\begin{proof}
Let $f_t$ be the induced one parameter family of $C^1$-diffeomorphisms of $X$.
By similar computations as \cite{Mag12}*{Exercise 20.7} (see also \cite{DeMasi22}*{Lemma 5.28, and (5.16)}), we find
\eq{
\frac{\rd}{\rd t}\mid_{t=0}\int_{f_t(\S)}F(\nu)\rd\mcH^n
=\int_\S F(\nu){\rm div}_{\S,F}X\rd\mcH^n,\\
\frac{\rd}{\rd t}\mid_{t=0}\int_{f_t(\widetilde\S)}F_\ast(\nu)\rd\mcH^n
=\int_{\widetilde\S} F_\ast(\nu){\rm div}_{\widetilde\S,F_\ast}X\rd\mcH^n.
}
On the other hand, by Lemmas \ref{Lem:S-F-F_ast} and \ref{Lem:1stvariation-widehat-S} we have
\eq{
&\frac{\rd}{\rd t}\mid_{t=0}\left(\int_{f_t(\S)}F(\nu)\rd\mcH^n+\int_{f_t(\widetilde\S)}F_\ast(\nu)\rd\mcH^n\right)\\
=&\int_{\widehat\S}{\rm div}_{\widehat\S}X\rd\mcH^n
-2\cos\theta\int_T{\rm div}_{\p\mfR^{n+1}_+}X\rd\mcH^n
=\int_{\widehat\S}H\left<X,\nu\right>\rd\mcH^n.
}
Combining the above computations, the assertion follows.
\end{proof}

%--------
\subsection{A priori estimates}
We record the following density-type estimate in \cite{JN23}*{Proposition 2.3}, which is essentially proven in \cite{MPS22}*{Lemma 2.1}:
\begin{proposition}\label{Prop-JN23-2.3}
Let $E\subset\mfR^{n+1}$ be a set of finite perimeter with $P(E)>0$ and $0<\beta<1$.
Then there exists a positive constant $c=c(n,\beta)$ such that
\eq{
r_{E,\beta}
\coloneqq
\sup\{r\in\mfR_+:\exists x\in\mfR^{n+1}\text{ with }\abs{B_r(x)\cap E}\geq\beta\abs{B_r(x)}\}
\geq c\frac{\abs{E}}{P(E)}.
}
\end{proposition}

%-----------
The next proposition concerns with the a priori estimates resulting from the capillary structure.

\begin{proposition}\label{Prop-JN23-Lem2.4}
Given $n\in\mbN^+$, $\theta\in(0,\pi)$,
let $\S\subset\overline{\mbR^{n+1}_+}$ be a compact $\theta$-capillary hypersurface.
Let $\Om$ denote the enclosed region of $\S$ with $\p\mbR^{n+1}_+$. %such that $\p\Om=\S\cup T$.

Given $\lambda\in\mbR_+$ and $1\leq C_0<\infty$, then there exists a positive constant $\bar C=\bar C(n,\theta,C_0)$ such that $P(\Om)\leq C_0$ and $\abs{\Om}\geq C_0^{-1}$ implies
\begin{enumerate}[label=(\roman*)]
\item 
The estimate on $\lambda$:
\eq{\label{ineq-JN23-Lem2.4-(i)}
\bar C^{-1}-\bar C\norm{H_\S-\lambda}_{L^1(\S)}
\leq\lambda
\leq \bar C+\bar C\norm{H_\S-\lambda}_{L^1(\S)}.
}
%----
\item 
For the family of connected components of $\Om$, say $\{\Om_i\}_{i\in J}$, if each connected component is adhering to $\p\mfR^{n+1}_+$, so that $\S_i\coloneqq\overline{\p\Om_i\cap\mfR^{n+1}_+}$ is a $\theta$-capillary hypersurface, then $\# J\leq\bar C(1+\norm{H_\S-\lambda}^n_{L^n(\S)})$,
and $d_{\rm ext}(\S_i)\leq \bar C(1+\norm{H_\S-\lambda}^{n-1}_{L^{n-1}(\S)})$.
\end{enumerate}
\end{proposition}
}

\begin{proof}
{
We shall work with the set $\widehat\Om=\Om\cup\widetilde\Om$.
First from \eqref{eq:widehatOm-widehatSigma} it is easy to see that
\eq{
\abs{\widehat\Om}\geq2C_0^{-1},\quad
\abs{\widehat\S}
\leq2C_0.
}
By the global isoperimetric inequality we have
\eq{
\abs{\widehat\Om}
\leq C(n)\abs{\widehat\S}^\frac{n+1}{n}
\leq C(n) C_0^\frac{n+1}{n}.
}
To prove \eqref{ineq-JN23-Lem2.4-(i)}, we observe that since the reflection map $\mathscr{R}$ is an isometry, we must have
\eq{\label{eq:H-lambda-H_sigma-lambda}
\norm{H-\lambda}_{L^1(\widehat\S)}
=\norm{H_\S-\lambda}_{L^1(\S)}
+\norm{H_{\widetilde\S}-\lambda}_{L^1(\widetilde\S)}
=2\norm{H_\S-\lambda}_{L^1(\S)}.
}

By Proposition \ref{Prop-JN23-2.3} (choosing $E=\widehat\Om$ and $\beta_0=\frac14$), there exists positive constants $r_0=r_0(n,C_0)$ and $R_0=R_0(n,C_0)$, such that for some $r_0\leq r\leq R_0$ and $x_r\in\mbR^{n+1}_+$, there holds
\eq{
\abs{B_r(x_r)\cap\widehat\Om}
=\frac14\abs{B_r(x_r)}.
%=\frac18\abs{B_r(x_r)}.
}

The relative isoperimetric inequality in ball (see e.g., \cite{Mag12}*{Proposition 12.37}), applying on $\widehat\Om\cap B_r(x_r)$ then yields
\eq{\label{ineq-relative-iso-truncated}
P(\widehat\Om;B_r(x_r))
\geq C(n)r^n
\geq C(n,C_0).
}

A direct computation, together with Lemma \ref{Lem:F-divergence-1st-variation}, shows that for any $X\in C_c^1(\mfR^{n+1},\mfR^{n+1})$,
\eq{\label{eq-JN23-(2-3)}
\lambda\int_{\widehat\Om}{\rm div}X\rd x
&=\int_{\widehat\S}\lambda\left<X,\nu\right>\rd\mcH^n\\
&=\int_{\widehat\S} H\left<X,\nu\right>\rd\mcH^n+\int_{\widehat\S}(\lambda-H)\left<X,\nu\right>\rd\mcH^n\\
&=\int_\S F(\nu){\rm div}_{\S,F} X\rd\mcH^n+\int_{\widetilde\S}F_\ast(\nu){\rm div}_{\widetilde\S,F_\ast} X\rd\mcH^n+\int_{\widehat\S}(\lambda-H)\left<X,\nu\right>\rd\mcH^n.
}

Define a radially symmetric vector field $X:\mbR^{n+1}\ra\mbR^{n+1}$ as
\eq{
X(x)
\coloneqq f(\abs{x-x_r})(x-x_r),
}
where $f:\mfR\ra\mfR$ is a non-increasing $C^1$-function to be specified later.

Clearly, $X\in C_c^1(\mfR^{n+1},\mfR^{n+1})$ and we wish to prove the first assertion by testing \eqref{eq-JN23-(2-3)} with such $X$. To this end, we compute on $\S$:
\eq{
\na X(x)
=&f(\abs{x-x_r}){\rm Id}+\frac{f'(\abs{x-x_r})}{\abs{x-x_r}}(x-x_r)\otimes(x-x_r),\\
{\rm div}X(x)
=&(n+1)f(\abs{x-x_r})+f'(\abs{x-x_r})\abs{x-x_r},
}
and by virtue of Proposition \ref{basic-F} ($ii$),
\eq{
&\left<\na F(\nu),(\na X)^\ast[\nu]\right>\\
=&F(\nu)\left(f(\abs{x-x_r})+f'(\abs{x-x_r})\frac{\abs{x-x_r}}{F(\nu)}\left<\frac{x-x_r}{\abs{x-x_r}},\nu\right>\left<\frac{x-x_r}{\abs{x-x_r}},\na F(\nu)\right>\right).
}
Recall that $\na F(\nu)=\nu-\cos\theta E_{n+1}$, taking \eqref{defn-m-M-F} into account,
we thus obtain
\eq{
G\coloneqq
&\frac1{F(\nu)}\left<\frac{x-x_r}{\abs{x-x_r}},\nu\right>\left<\frac{x-x_r}{\abs{x-x_r}},\na F(\nu)\right>\\
=&\frac1{F(\nu)}\left<\frac{x-x_r}{\abs{x-x_r}},\nu\right>^2-\frac{\cos\theta}{F(\nu)}\left<\frac{x-x_r}{\abs{x-x_r}},\nu\right>\left<\frac{x-x_r}{\abs{x-x_r}},E_{n+1}\right>\\
\in&[-\frac{\abs{\cos\theta}}{1-\abs{\cos\theta}},\frac{1+\abs{\cos\theta}}{1-\abs{\cos\theta}}],
}
from which we deduce,
at any $x\in\S$,
\eq{
{\rm div}_{\S,F}X
=F(\nu)\left(nf(\abs{x-x_r})+\abs{x-x_r}f'(\abs{x-x_r})(1-G)\right),
}
and satisfies the estimates (recall that $f'\leq0$)
\eq{
&{\rm div}_{\S,F}X
\geq F(\nu)\left(nf(\abs{x-x_r}+\frac1{1-\abs{\cos\theta}}\abs{x-x_r}f'(\abs{x-x_r}))\right),\\
&{\rm div}_{\S,F}X
\leq F(\nu)\left(nf(\abs{x-x_r}-\frac{2\abs{\cos\theta}}{1-\abs{\cos\theta}}\abs{x-x_r}f'(\abs{x-x_r}))\right).
}

Now we construct $f$, let us first look at the one variable function $g(t)=\frac1{t-m}$.
A direct computation then shows that, when $t>m$, $ng(t)+\frac1{1-\abs{\cos\theta}}tg'(t)\geq0$ is equivalent to
\eq{
m
\leq(1-\frac1{n(1-\abs{\cos\theta})})t.
}
Let us set
\eq{
m_0
\coloneqq(1-\frac1{1-\abs{\cos\theta}})\frac{5r}{2(1-\abs{\cos\theta})}
\leq(1-\frac1{n(1-\abs{\cos\theta})})\frac{5r}{2(1-\abs{\cos\theta})},
}
and define $f(t)=\frac1{t-m_0}$ on $[\frac{5r}{2(1-\abs{\cos\theta})},\infty)$, it follows
from the above observation that $nf(t)+\frac1{1-\abs{\cos\theta}}tf'(t)\geq0$ on this interval.
Also, it is easy to see that
\eq{
f(\frac{5r}{2(1-\abs{\cos\theta})})
=&\frac{2(1-\abs{\cos\theta})^2}{5r},\\
-f'(\frac{5r}{2(1-\abs{\cos\theta})})
=&\frac{4(1-\abs{\cos\theta})^4}{25r^2},
%\leq\frac{(1-\abs{\cos\theta})^4}{4r^2}.
}
which in turn implies that, on the interval $[\frac{5r}{2(1-\abs{\cos\theta})},\infty)$,
\eq{\label{ineq-t-f'(t)-1}
-tf'(t)
\leq(1-\abs{\cos\theta})nf(t)
\leq(1-\abs{\cos\theta})nf(\frac{5r}{2(1-\abs{\cos\theta})})
=\frac{2n(1-\abs{\cos\theta})^3}{5r}.
}

To proceed, we define a non-increasing $C^1$-function $f:\mfR\ra\mfR$ by
\eq{
f(t)
=
\begin{cases}
\frac{4(1-\abs{\cos\theta})^2+(1-\abs{\cos\theta})^3}{10r},\text{ }&t\leq\frac{3r}{2(1-\abs{\cos\theta})},\\
\frac1{t-m_0},\text{ }&t\geq\frac{5r}{2(1-\abs{\cos\theta})},
\end{cases}
}
and for which the condition $-f'(t)\leq\frac{4(1-\abs{\cos\theta})^4}{25r^2}$ holds on $[\frac{3r}{2(1-\abs{\cos\theta})},\frac{5r}{2(1-\abs{\cos\theta})}]$.
In fact, the existence of such $f$ is ensured by the following computation
\eq{
\frac{f(\frac{3r}{2(1-\abs{\cos\theta})})-f(\frac{5r}{2(1-\abs{\cos\theta})})}{\frac{r}{1-\abs{\cos\theta}}}
=\frac{(1-\abs{\cos\theta})^4}{10r^2}
<\frac{4(1-\abs{\cos\theta})^4}{25r^2}.
}

Now we verify that
$nf(t)+\frac1{1-\abs{\cos\theta}}tf'(t)\geq0$ on $[\frac{3r}{2(1-\abs{\cos\theta})},\frac{5r}{2(1-\abs{\cos\theta})}]$, which can be done through the following computations: on such an interval,
\eq{
nf(t)
\geq nf(\frac{5r}{2(1-\abs{\cos\theta})})
=\frac{2n(1-\abs{\cos\theta})^2}{5r},
}
while
\eq{
-\frac1{1-\abs{\cos\theta}}tf'(t)
\leq \frac{4(1-\abs{\cos\theta})^3}{25r^2}t
\leq\frac{2(1-\abs{\cos\theta})^2}{5r}
\leq\frac{2n(1-\abs{\cos\theta})^2}{5r}.
}
We have two by-products of this estimate.
First, the estimate \eqref{ineq-t-f'(t)-1} is extended to hold on $[\frac{3r}{2(1-\abs{\cos\theta})},\frac{5r}{2(1-\abs{\cos\theta})}]$.
Second, we have $nf(t)+\frac1{1-\abs{\cos\theta}}tf'(t)\geq0$ on $[\frac{3r}{2(1-\abs{\cos\theta})},\infty]$, and because $f'\leq0$, we find
\eq{
(n+1)f(t)+tf'(t)
>nf(t)+\frac1{1-\abs{\cos\theta}}tf'(t)
\geq0,\quad t\in[0,\infty].
}

Since the above computations hold for $F=F_\theta$ on $\S$, we could therefore use the same argument to show that these computations, especially \eqref{ineq-t-f'(t)-1}, also hold for $F_\ast=F_{\pi-\theta}$ on $\widetilde\S$.

By setting $t(x)=\abs{x-x_r}$, $h_1=h_1(n,\theta)\coloneqq(n+1)\frac{4(1-\abs{\cos\theta})^2+(1-\abs{\cos\theta})^3}{10}>0$, we have shown that ${\rm div}X=(n+1)f(t)+f'(t)t$ satisfies:
$0<{\rm div}X\leq\frac{h_1}{r}$ everywhere and ${\rm div}X=\frac{h_1}{r}$ in $B_r(x_r)$, thus obtaining
\eq{\label{ineq-JN23-(2-4)}
\frac{h_1}{4R_0}\abs{B_{r_0}(x_r)}
\leq&\frac{h_1}{4r}\abs{B_r(x_r)}
=\frac{h_1}{r}\abs{B_r(x_r)\cap\widehat\Om}
\\
\leq&\int_{\widehat\Om}{\rm div}X\rd x
\leq\frac{h_1}{r}\abs{\widehat\Om}
\leq\frac{C(n,\theta)}{r_0}C_0^\frac{n+1}{n}.
}

On the other hand,
combining the estimates above, especially \eqref{ineq-t-f'(t)-1}, we find that
\eq{
0
\leq{\rm div}_{\S,F}X
\leq M_F\left(nf(0)+\frac{4n\abs{\cos\theta}(1-\abs{\cos\theta})^2}{5r}\right)
\eqqcolon \frac{h_2(n,\theta)}r
}
on $\S$, and
\eq{
{\rm div}_{\S,F}X
=nF(\nu)f(0)
\geq\frac{h_3(n,\theta)}{r}
}
on $\S\cap B_r(x_r)$;
note also that these estimates also hold for ${\rm div}_{\widetilde\S,F}X$ on $\widetilde\S$,
from which we deduce
\eq{\label{ineq-JN23-(2-5)}
\frac{C(n,\theta,C_0)}{R_0}&\overset{\eqref{ineq-relative-iso-truncated}}{\leq}\frac{h_3P(\widehat\Om; B_r(x_r))}{r}\\
\leq&
\int_\S{\rm div}_{\S,F}X\rd\mcH^n+\int_{\widetilde\S}{\rm div}_{\widetilde\S,F_\ast}X\rd\mcH^n
\leq\frac{h_2}{r}\abs{\widehat\S}
\leq\frac{C(n,\theta,C_0)}{r_0}.
}

To complete the proof of \eqref{ineq-JN23-Lem2.4-(i)}, we recall that on $t>\frac{5r}{2(1-\abs{\cos\theta})}$,
\eq{
tf(t)
=1+\frac{m_0}{t-m_0}
\leq1+m_0f(\frac{5r}{2(1-\abs{\cos\theta})})
=1-\abs{\cos\theta},
}
while on $[0,\frac{5r}{2(1-\abs{\cos\theta})}]$,
\eq{
tf(t)
\leq\frac{5r}{2(1-\abs{\cos\theta})}f(0)
=C(\theta),
}
inferring that $\abs{X}\leq C(\theta)$ on $\mfR^{n+1}$.
This fact,
in conjunction with \eqref{eq:H-lambda-H_sigma-lambda}, \eqref{eq-JN23-(2-3)}, \eqref{ineq-JN23-(2-4)}, \eqref{ineq-JN23-(2-5)}, leads to \eqref{ineq-JN23-Lem2.4-(i)}.

}

%-------
To prove {\bf(ii)}, we consider the following two cases separately:

{\bf Case 1.} $n=1$.

For any connected component $\Om_i$, up to a translation along $\p\mfR^{n+1}_+$, we may assume that the origin $O\in\p\S_i$.
Testing \eqref{formu-1st-variation} with the position vector field $X(x)=x$, we find: for $\theta\in(0,\frac\pi2]$,
\eq{
(1-\cos\theta)\abs{\S_i}
\overset{\eqref{ineq-T<=Sigma}}{\leq}&\abs{\S_i}-\cos\theta\abs{T_i}
=\int_{\S_i}{\rm div}_{\S_i}X\rd\mcH^1-\cos\theta\int_{T_i}{\rm div}_{\p\mfR^{n+1}_+}X\rd\mcH^1\\
=&\int_{\S_i}\left<x,H_{\S_i}(x)\right>\rd\mcH^1
\leq \norm{H_{\S_i}}_{L^1(\S_i)}\abs{\S_i},
}
where we have used the fact that $O\in\p\S_i$, and hence $\abs{x}\leq \mcH^1(\S_i)=\abs{\S_i}$ for any $x\in\S_i$.
The case that $\theta\in(\frac\pi2,\pi)$ follows similarly, because $\abs{\S_i}-\cos\theta\abs{T_i}\geq\abs{\S_i}$.
It is then easy to see that
\eq{\label{ineq-P-Omega_i-1}
\min\{1-\cos\theta,1\}
\leq\norm{H_{\S_i}}_{L^1(\S_i)}
\leq\norm{H_{\S_i}-\lambda}_{L^1(\S_i)}+\lambda P(\Om_i),
}
which, in conjunction with {\bf(i)} and the fact that $P(\Om)\leq C_0$, yields
\eq{
\# J
\leq\frac{\norm{H_\S-\lambda}_{L^1(\S)}+\lambda P(\Om)}{\min\{1-\cos\theta,1\}}
\leq C(1+\norm{H_\S-\lambda}_{L^1(\S)}).
}
The upper bound on diameters of $\Om_i$ follows from \eqref{ineq-Topping-capillary} and the fact that $\abs{\S_i}
<P(\Om)\leq C_0$.

{\bf Case 2.} $n\geq2$.

This case can be handled similarly as \cite{JN23}*{(2-6), (2-7)}, once we have applied \eqref{ineq-Michael-Simon} on each connected component $\Om_i$ with $f\equiv1$ on $\S_i$ and H\"older's inequality to find
\eq{\label{ineq-P-Omega_i-n}
\sigma(n,\theta)^{-1}
\leq\norm{H_{\S_i}}_{L^n(\S_i)}
\leq\norm{H_{\S_i}-\lambda}_{L^n(\S_i)}+\lambda P(\Om_i)^\frac1n.
}
\end{proof}

\begin{remark}\label{Rem-apriori}
\normalfont
\begin{enumerate}
%----
\item In Proposition \ref{Prop-JN23-Lem2.4}, the condition that $P(\Om)$ is bounded from above is equivalent to requiring a similar upper bound on $\abs{\S}$, since on one hand $\abs{\S}<P(\Om)$, on the other hand we have $P(\Om)\leq 2\abs{\S}$, thanks to \eqref{ineq-T<=Sigma};
%--------
\item We have the apriori estimate
\eq{\label{ineq-diam-d_ext}
{\rm diam}(\Om_i)
\leq\frac32d_{\rm ext}(\S_i).
}
To see this,
we define
\eq{\label{defn-d_m}
d_m
\coloneqq\max_{x\in\overline{T}}{\rm dist}(x,\S).
}
From the definitions of $d_m$ and $d_{\rm ext}$, we clearly have
\eq{
2d_m
\leq d_{\rm ext}(\p\S_i)
\leq d_{\rm ext}(\S_i)
}
for each connected component $\Om_i$.
On the other hand, we easily infer from the triangle inequality that ${\rm diam}(\Om_i)=\max_{x,y\in\overline{\Om}_i}\abs{x-y}\leq d_{\rm ext}(\S_i)+d_m$, \eqref{ineq-diam-d_ext} then follows.

\end{enumerate}

\end{remark}
%------
\subsection{Density-type estimate}
We end this section with the following density-type estimate, which generalizes \cite{JN23}*{Lemma 3.2} from closed hypersurfaces to the capillary setting.
\begin{proposition}\label{Prop-density}
Given $n\in\mbN^+$, $\theta\in(0,\pi)$,
let $\S\subset\overline{\mbR^{n+1}_+}$ be a compact $\theta$-capillary hypersurface.
Let $\Om$ denote the enclosed region of $\S$ with $\p\mbR^{n+1}_+$. %such that $\p\Om=\S\cup T$.

For $n\in\mbN$, there exists a positive constant $\de_{n,\theta}=\de_{n,\theta}(n,\theta)$ such that for any $\lambda\in\mfR_+$, if $\norm{H_\S-\lambda}_{L^n(\S)}\leq\de_{n,\theta}$, then
\eq{
\de_{n,\theta}
\leq\frac{\mcH^n(\S\cap \mcW_{r}(x))}{r^n}
}
for every $x\in\S$ and $0<r\leq\frac{\de_{n,\theta}}\lambda$.

\iffalse
{
Assume in addition that $P(\Om)\leq C_0$ and $\abs{\Om}\geq C_0^{-1}$ for some constant $1\leq C_0<\infty$;
each connected components of $\Om$, say $\{\Om_i\}_{i\in J}$, is adhering to $\p\mfR^{n+1}_+$ with volume bounded from below by some positive constant $C=C(n,\theta,C_0)$,
then for $n=1$, the above statement holds with $\de_{1,\theta}$ depending additionally on $C_0$.
}
\fi
\end{proposition}
\begin{proof}
For every $x\in\S$, we define
\eq{
V(x,r)
\coloneqq\mcH^n(\S\cap \mcW_{r}(x))
=\mcH^n\llcorner\S(\mcW_{r}(x)),
}
for simplicity, we omit the argument $x$ and denote $V(x,r)$ by $V(r)$.

Since $\S$ is compact, we see that the function $V$ is bounded and $\mcH^n\llcorner\S$ is a Radon measure.
Therefore we easily see that $V(r)$ is non-decreasing on $[0,\infty)$, the derivative $V'(r)$ is well-defined for almost every $r\in[0,\infty)$, and
\eq{
\int_{r_1}^{r_2}V'(\rho)\rd\rho
\leq V(r_2)-V(r_1)\text{ for any }0\leq r_1<r_2.
}
Moreover,
we have that $\mcH^n(\S\cap\p\mcW_{r}(x))=0$ for almost every $r\in(0,\infty))$ due to \cite{Mag12}*{Proposition 2.16}.
Fix any such $r$ and for any $h\in\mbR_+$ with $\mcH^n(\p\mcW_{r+h}(x)\cap\S)=0$, we define a Lipschitz cut-off function $f_h:\mbR^{n+1}\ra\mbR$ by setting
\eq{
f_h(y)
=
\begin{cases}
1,\quad&y\in \mcW_{r}(x),\\
1-\frac{1}{h}(F^o(y-x)-r),\quad&y\in \mcW_{r+h}(x)\setminus \mcW_{r}(x),\\
0,\quad&y\notin \mcW_{r+h}(x).
\end{cases}
}
A direct computation shows that on $\mcW_{r+h}(x)\setminus \mcW_{r}(x)$, 
\eq{\label{ineq-na^S-f_h}
\abs{\na^\S f_h(y)}
\leq\frac1h\abs{\na F^o(y-x)}
\leq\frac1{m_Fh}F(\na F^o(y-x))
=\frac1{m_Fh}
=\frac1{h(1-\abs{\cos\theta})}.
}

{\bf Case 1. }$n\geq2$.

\noindent{\bf Step 1. We prove that there exists $\sigma(n,\theta)>0$ such that for a.e. $0<r<\frac{\de}{\lambda}$,
\eq{\label{ineq-V(r)-V'(r)}
V(r)^\frac{n-1}{n}
\leq \tilde\sigma(n,\theta)\left(V'(r)+\norm{H_\S-\lambda}_{L^n(\S)}V(r)^\frac{n-1}{n}+\lambda V(r)\right).
}
}

Using a standard smooth approximation argument of the Lipschitz function $f_h$, we may exploit the Michael-Simon-type inequality \eqref{ineq-Michael-Simon} if $n\geq2$, with $f_h$ and \eqref{ineq-na^S-f_h} to obtain
\eq{
V(r)^\frac{n-1}{n}
\leq\sigma(n,\theta)
\left(\frac{V(r+h)-V(r)}{h(1-\abs{\cos\theta})}+\norm{f_hH}_{L^1(\S)}\right).
}
Define $\tilde\sigma(n,\theta)=\frac{\sigma(n,\theta)}{1-\abs{\cos\theta}}$,
we may choose a sequence $(h_k)_k$ such that $h_k\ra0^+$ and $\mcH^n(\S\cap\p \mcW_{r+h_k}(x))=0$ for each $h_k$, thus by letting $k\ra\infty$ in the above inequality, we arrive at
\eq{
V(r)^\frac{n-1}{n}
\leq&\tilde\sigma(n,\theta)(V'(r)+\norm{H}_{L^1(\S\cap \mcW_r(x))})\\
\leq&\tilde\sigma(n,\theta)(V'(r)+\norm{H-\lambda}_{L^1(\S\cap \mcW_r(x))}+\lambda V(r)),
}
\eqref{ineq-V(r)-V'(r)} then follows from H\"older's inequality, this completes the first step.

\noindent{\bf Step 2. We finish the proof of this case, i.e., we prove that for a.e. $0<r<\frac{\de_{n,\theta}}\lambda$, there holds
\eq{
V(r)\geq\de_{n,\theta} r^n.
}
}

If this is false, namely, if $V(r)<\de_{n,\theta} r^n$ for some fixed $0<r<\frac{\de_{n,\theta}}\lambda$, then we trivially have
\eq{
\lambda V(\rho)^\frac1n
\leq \lambda V(r)^\frac1n
\leq\de_{n,\theta}^{\frac{n+1}n}
}
for every $0<\rho<r$.

Rearranging \eqref{ineq-V(r)-V'(r)}, we deduce for a.e. $0<\rho<r$
\eq{
\left(\frac{\tilde\sigma^{-1}(n,\theta)-\norm{H_\S-\lambda}_{L^n(\S)}}{V(\rho)^\frac1n}-\lambda\right)V(\rho)
\leq V'(\rho),
}
if we choose $\de_{n,\theta}<\frac1{2\tilde\sigma(n,\theta)}$, then we find
\eq{
\left(\frac1{2\tilde\sigma(n,\theta)}-\de_{n,\theta}^{\frac{n+1}n}\right)V(\rho)^{1-\frac1n}
\leq\frac{1}{2\tilde\sigma(n,\theta)}V(\rho)^{1-\frac1n}-\lambda V(\rho)
\leq V'(\rho).
}
After further restricting $\de_{n,\theta}<\min\left\{1,\frac1{4\tilde\sigma(n,\theta)},\left(\frac1{4n\tilde\sigma(n,\theta)}\right)^n\right\}$, the above inequality gives
\eq{
\frac1{4\tilde\sigma(n,\theta)}
\leq \frac{V'(\rho)}{V(\rho)^{1-\frac1n}}.
}
Integrating this over $(0,r)$, we obtain
\eq{
\left(\frac1{4n\tilde\sigma(n,\theta)}\right)^nr^n
\leq V(r),
}
a contradiction to $V(r)<\de_{n,\theta} r^n$.

{\bf Case 2.} $n=1$.

The proof is essentially the same as {\bf Case 1}, expect that we use \eqref{ineq-Michael-Simon-n=1} in this case, and obtain
\eq{
V(r)
\leq\sigma(n=1,\theta)V(r+h)\left(\frac{V(r+h)-V(r)}{h(1-\abs{\cos\theta})}+\norm{f_hH}_{L^1(\S)}\right).
}
Arguing as {\bf Case 1}, and
note that $V(r)$ and $V(r+h_k)$ cancels if we let $h_k\ra0^+$.
In particular, we arrive at
\eq{
1
\leq\tilde\sigma(n=1,\theta)\left(V'(r)+\norm{H-\lambda}_{L^1(\S)}+\lambda V(r)\right).
}
The rest of the proof is just the same.
\end{proof}
%=========

\section{Quantitative Alexandrov theorem}\label{Sec-5}

%------------
\subsection{Shifted distance function}

In \cite{XZ23} we introduce the following essential tool, called the shifted distance function, which is found useful and shown to be the ``correct'' distance function that one should study when dealing with capillary problem in the half-space.

Given $\theta\in(0,\pi)$, and a (possibly not connected) bounded open set $\Om\subset\overline{\mfR^{n+1}_+}$ which is adhering to $\p\mfR^{n+1}_+$, whose relative boundary $\S=\overline{\p\Om\cap\mfR^{n+1}_+}$ is a compact $C^2$-hypersurface,
let $u:\mfR^{n+1}\to \mfR$ be the \textit{distance function with respect to $\S$}, defined as
\eq{
    u(y)=\sup_{r\geq0}\{r:B_r(y)\cap\S
    =\varnothing\}.
}
and $u_{\theta}:\mfR^{n+1}\to \mfR$ be the \textit{shifted distance function with respect to $\S$ and $\theta$}, defined as
\eq{\label{defn-weighteddistance}
    u_{\theta}(y)
    =\sup_{r\geq0}\{r:\mcW_r(y)\cap \S
    =\varnothing\},
}
which is a Lipschitz function on $\mfR^{n+1}$ with Lipschitz constant at most $\frac{1}{1-\abs{\cos\theta}}$, see \cite{XZ23}*{Lemma 3.4}.

One also sees from definition that
\eq{\label{eq-dist-theta}
    u_\theta(y)
    =u(y-u_\theta(y)\cos\theta E_{n+1})
}
and for any $0<r<u_{\theta}(y)$, there holds
\eq{\label{ineq-weighted-distance}
   u(y-r\cos\theta E_{n+1})>r.
}
See Fig. \ref{figure2}.	\begin{figure}[H]
	\centering
	\includegraphics[height=7cm,width=12cm]{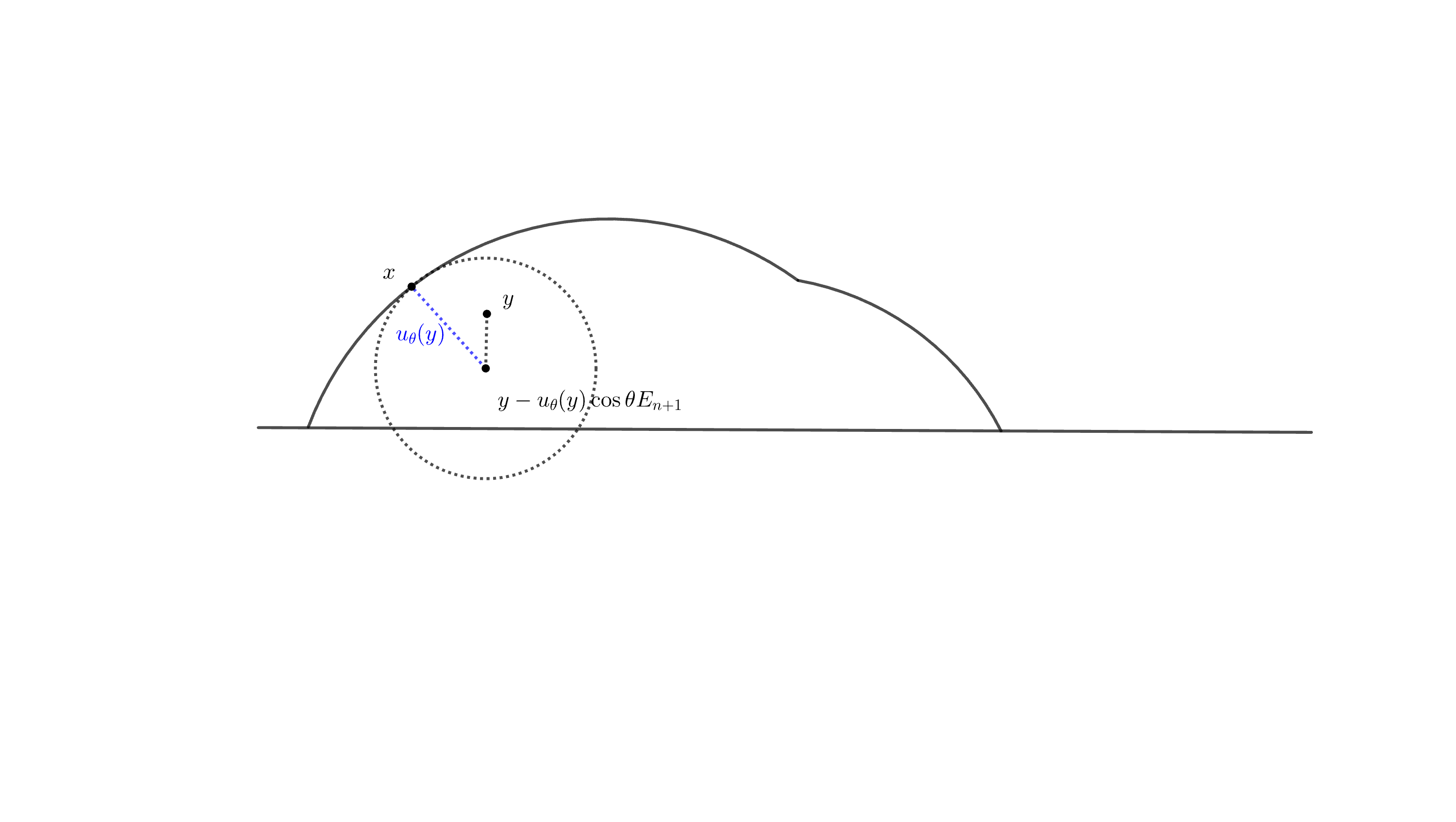}
	\caption{shifted distance function}
	\label{figure2}
\end{figure}
In view of Section 
\ref{Sec-2-1-1}, the shifted distance function is a natural device to be studied, since it can be equivalently characterized as
\eq{
u_\theta(y)
=\min_{z\in\S}F^o(z-y),
}
and amounts to be the anisotropic counterpart of the Euclidean distance function $u$.

For $s>0$, we define the super level-set and level-set of $u_{\theta}$ in $\overline\Omega$ by
\eq{\label{defn-Omegas}
    \Omega_s\coloneqq\{y\in\overline\Omega:u_{\theta}(y)>s\},
    \quad
    \pr\Om_s\coloneqq\{y\in\overline\Omega:u_{\theta}(y)=s\}.
}
%-----
We define for every $s\geq0$
\eq{
\S_s
\coloneqq\left\{x\in\S:\text{ there exists }y\in\pr\Om_s\text{ such that }u_\theta(y)=s\text{ attains at }x\right\};
}
in other words, for any $x\in\S_s$, there exists $y\in\pr\Om_s$ such that $x\in \mcW_s(y)\cap\S$.
Clearly, $\S_0=\S$, and for any $0\leq s_1<s_2$, we have the inclusion $\S_{s_2}\subset\S_{s_1}$.

%-----------
\subsection{Area and volume estimates in terms of small deficit}
\begin{proposition}\label{Prop-JN23-Prop3.3}
Given $n\in\mbN^+$, $\theta\in(0,\pi)$,
let $\S\subset\overline{\mbR^{n+1}_+}$ be a compact $\theta$-capillary hypersurface.
Let $\Om$ denote the enclosed region of $\S$ with $\p\mbR^{n+1}_+$. %such that $\p\Om=\S\cup T$.

Given $\lambda\in\mbR_+$ and $1\leq C_0<\infty$, if $P(\Om)\leq C_0$ and $\abs{\Om}\geq C_0^{-1}$, then for any $0<r<R=\frac{n}\lambda$, there exist $\de=\de(n,\theta,C_0)>0$, $C=C(n,\theta,C_0)$, such that if 
\eq{
\norm{H_\S-\lambda}_{L^n(\S)}\leq\de,
}
then there hold
\eq{\label{ineq-Om_r}
\Abs{\abs{\Om_r}-\frac{\abs{\Om}}{R^{n+1}}(R-r)^{n+1}}
\leq C(n,\theta,C_0)\norm{H_\S-\lambda}_{L^n(\S)}.
}
and
\eq{\label{ineq-Sigmma-Sigma_r}
\int_{\S\setminus\S_r}F(\nu)\rd\mcH^n
\leq\frac{C(n,\theta,C_0)}{(R-r)^{n+1}}\norm{H_\S-\lambda}_{L^n(\S)}.
}
Moreover, for $0<\rho<r$, there holds
\eq{\label{ineq-Om_r-rho}
\Abs{\Abs{(\Om_r+\mcW_{\rho})\cap\overline{\mfR^{n+1}_+}}-\frac{\abs{\Om}}{R^{n+1}}(R-(r-\rho))^{n+1}}
\leq\frac{C(n,\theta,C_0)}{(R-r)^{n+1}}\norm{H_\S-\lambda}_{L^n(\S)}.
}
\end{proposition}
\begin{proof}
For simplicity, we assume that the principal curvatures of $\S$ at $x$, say $\{\kappa_i(x)\}_{i=1,\ldots,n}$, are indexed in the increasing order.
Define the set of ``good'' points of $\S$ as
\eq{
\S_G\coloneqq\{x\in\S:\abs{H_\S(x)-\lambda}<\frac12\lambda\},
}
and correspondingly the set of points of $\S$ at which we only expect "bad" behavior
\eq{
\S_B\coloneqq\S\setminus\S_G.
}

On one hand, for some $\de<1$ to be specified later, we exploit Proposition \ref{Prop-JN23-Lem2.4} to see that $\Om$ can be decomposed to $\#J\leq2\bar C$ connected components, each of which has diameter upper bound $3\bar C$ by virtue of Remark \ref{Rem-apriori} (2).
Therefore, we may prove the proposition componentwise and assume that $\Om$ is connected.
On the other hand, we also notice that a simple application of the triangle inequality yields: for any $y\in\overline\Om$,
\eq{
u_\theta(y)
\leq\abs{\cos\theta}u_\theta(y)+\abs{y-x},
}
where $x\in\S$ is the point at which the shifted distance $u_\theta(y)$ is attained,
and hence
\eq{
u_\theta(y)
\leq\frac{3\bar C}{1-\abs{\cos\theta}}\eqqcolon\tilde R(n,\theta,C_0)=\tilde R.
}

\noindent{\bf Step 1. Quantify the Heintze-Karcher-type inequality in the spirit of \cite{JWXZ23}.
}

In view of the introduction and \cite{JN23}, we define
\eq{
Z_G
\coloneqq\left\{(x,t)\in\S_G\times[0,\infty):0\leq t\leq\frac{1}{\kappa_n(x)}\right\},
}
which is clearly well-defined since $\kappa_n(x)\geq\frac{H_\S(x)}n\geq\frac\lambda{2n}>0$.
For $\S_B$, we first further decompose it to be
\eq{
\S'_B
\coloneqq\{x\in\S_B:\kappa_n(x)\leq\frac{1}{\tilde R}\},\text{ and }
\S''_B
\coloneqq\{x\in\S_B:\kappa_n(x)>\frac1{\tilde R}\},
}
then set
\eq{
Z'_B
\coloneqq&\S_B'\times[0,\tilde R],\\
Z''_B
\coloneqq&\left\{(x,t)\in\S_B''\times[0,\infty):0\leq t\leq\frac1{\kappa_n(x)}\right\},\\
Z_B
\coloneqq& Z'_B\cup Z''_B.
}
Clearly, $Z_G,Z_B$ are disjoint and bounded, and we claim that
\eq{\label{ineq-Om-zeta_theta}
\Om\subset\zeta_F(Z_G\cup Z_B).
}
In fact, for any $y\in\Om$ such that $r\coloneqq u_\theta(y)$ and for any $x\in\S$ at which $u_\theta(y)$ is attained, we may first infer from \cite{JWXZ23}*{Proof of Theorem 1.2, Case 2} that $x$ cannot be on $\p\S$, then from the definition of $u_\theta(y)$ that $y=x-r\nu_F(x)$, and finally from \cite{JWXZ23}*{Proof of Theorem 1.2, Case 1} that $\kappa_n(x)\leq\frac1r$.
The claim is thus validated by the fact that $r=u_\theta(y)\leq\tilde R$.

Next, we conduct a computation similar to that presented in the introduction to find
\eq{
\abs{\Om}
\leq&\abs{\zeta_F(Z_G)}+\abs{\zeta_F(Z_B)}
\leq\int_{\zeta_F(Z_G)}\mcH^0(\zeta_F^{-1}(y))\cap Z_G)\rd y+\abs{\zeta_F(Z_B)}\\
=&\int_{Z_G}{\rm J}^{Z_G}\zeta_F\rd\mcH^{n+1}+\abs{\zeta_F(Z_B)}\\
=&\int_{\S_G}\int_0^\frac{1}{\kappa_n(x)}F(\nu)\prod_{i=1}^n(1-t\kappa_i(x))\rd t\rd\mcH^n(x)+\abs{\zeta_F(Z_B)}\\
\leq&\int_{\S_G}F(\nu)\int_0^\frac{1}{\kappa_n(x)}\left(1-t\frac{H_\S(x)}{n}\right)^n\rd t\rd\mcH^n+\abs{\zeta_F(Z_B)}\\
\leq&\int_{\S_G}F(\nu)\int_0^\frac{n}{H_\S(x)}\left(1-t\frac{H_\S(x)}{n}\right)^n\rd t\rd\mcH^n+\abs{\zeta_F(Z_B)}\\
=&\frac{n+1}n\int_{\S_G}\frac{F(\nu)}{H_\S}\rd\mcH^n+\abs{\zeta_F(Z_B)}.
}
Let us keep track of the errors that arise each time we estimate with an inequality in the above argument. Precisely, we set
\eq{
N_1
\coloneqq&\abs{\zeta_F(Z_G)\setminus\Om},\\
N_2
\coloneqq&\int_{\zeta_F(Z_G)}\abs{\mcH^0(\zeta_F^{-1}(y)\cap Z_G)-1}\rd y,\\
N_3
\coloneqq&\int_{\S_G}F(\nu)\int_0^\frac{1}{\kappa_n(x)}\Abs{\left(1-t\frac{H_\S}n\right)^n-\prod_{i=1}^n(1-t\kappa_i(x))}\rd t\rd\mcH^n\\
N_4
\coloneqq&\int_{\S_G}F(\nu)\int_{\frac1{\kappa_n(x)}}^\frac{n}{H_\S(x)}\Abs{\left(1-t\frac{H_\S}n\right)^n}\rd t\rd\mcH^n.
}
The Heintze-Karcher-type inequality can then be quantified as
\eq{\label{ineq-hk-N_1-N_4}
\abs{\Om}
\leq\frac{n+1}n\int_{\S_G}\frac{F(\nu)}{H_\S}\rd\mcH^n+\abs{\zeta_F(Z_B)}-N_1-N_2-N_3-N_4.
}

\noindent{\bf Step 2. Quantify the Heintze-Karcher-type inequality using \eqref{eq-JN23-(2-3)}.
}

Using H\"older inequality, we find
\eq{
\norm{H_\S-\lambda}_{L^1(\S)}
\leq\abs{\S}^\frac{n-1}{n}\norm{H_\S-\lambda}_{L^n(\S)}
<C(n,C_0)\de.
}
By virtue of Proposition \ref{Prop-JN23-Lem2.4}, we could further decrease $\de$ to obtain
\eq{
0<\frac1{2\bar C}
\leq\lambda\leq2\bar C,
}
%$\lambda\geq\frac1{2\bar C}>0$,
and hence we may estimate the area of the "bad" set by H\"older's inequality
\eq{\label{ineq-JN23-(3-3)}
\mcH^n(\S_B)
\leq\frac2\lambda\int_\S\abs{H_\S(x)-\lambda}\rd\mcH^n
\leq C(n,\theta,C_0)\norm{H_\S-\lambda}_{L^n(\S)},
}
while for the "good" set,
\eq{\label{ineq-hk-H_S-lambda}
&\frac{n}{n+1}\int_{\S_G}\frac{F(\nu)}{H_\S}\rd\mcH^n
=\frac{n}{n+1}\int_{\S_G}\frac{F(\nu)}{\lambda}+F(\nu)(\frac1{H_\S}-\frac1\lambda)\rd\mcH^n\\
\overset{\eqref{ineq-JXZ23-(20)}}{\leq}&\frac{nP_F(\Om)}{(n+1)\lambda}+\frac{n(1+\abs{\cos\theta})}{n+1}\frac{2}{\lambda^2}\int_\S\abs{H_\S-\lambda}\rd\mcH^n\\
\leq&\frac{nP_F(\Om)}{(n+1)\lambda}+C(n,\theta,C_0)\norm{H_\S-\lambda}_{L^n(\S)},
}
where we have used in the first inequality the fact that $F(\nu)>0$, and hence
\eq{
\int_{\S_G}F(\nu)\rd\mcH^n\leq\int_{\S}F(\nu)\rd\mcH^n
=P_F(\Om).
}

On the other hand, fix any $x\in\p\S$, by testing \eqref{eq-JN23-(2-3)} with $X(y)=y-x$, we get
\eq{
nP_F(\Om;\mfR^{n+1}_+)
=(n+1)\lambda\abs{\Om}+\int_\S(H_\S-\lambda)\left<y-x,\nu(y)\right>\rd\mcH^n,
}
it follows from Proposition \ref{Prop-JN23-Lem2.4} and our choice of $\de$ that
\eq{\label{ineq-JN23-(3-4)}
\Abs{\frac{nP_F(\Om;\mfR^{n+1}_+)}{(n+1)\lambda}-\abs{\Om}}
\leq C(n,\theta,C_0)\norm{H_\S-\lambda}_{L^n(\S)},
}
which, in conjunction with \eqref{ineq-hk-H_S-lambda}, implies
\eq{
\frac{n}{n+1}\int_{\S_G}\frac{F(\nu)}{H_\S}\rd\mcH^n-\abs{\Om}
\leq C(n,\theta,C_0)\norm{H_\S-\lambda}_{L^n(\S)}.
}
Substituting this back into \eqref{ineq-hk-N_1-N_4}, we obtain the estimate of the error terms:
\eq{\label{ineq-N_1+N_2+N_3+N_4}
N_1+N_2+N_3+N_4
\leq\abs{\zeta_F(Z_B)}+C(n,\theta,C_0)\norm{H_\S-\lambda}_{L^n(\S)}.
}

\noindent{\bf Step 3. We show that
\eq{\label{ineq-JN23-(3-10)}
\abs{\zeta_F(Z_B)}
\leq C(n,\theta,C_0)\norm{H_\S-\lambda}_{L^n(\S)}.
}
}

Note that a direct consequence of this step is, the error terms $N_1+\ldots N_4$ will be also controlled by $\norm{H_\S-\lambda}_{L^n(\S)}$ thanks to \eqref{ineq-N_1+N_2+N_3+N_4}, that is,
\eq{\label{ineq-JN23-(3-12)}
N_1+N_2+N_3+N_4
\leq C(n,\theta,C_0)\norm{H_\S-\lambda}_{L^n(\S)}.
}

To prove \eqref{ineq-JN23-(3-10)}, we use the definitions of $Z'_B, Z_B''$, and the area formula to find
\eq{\label{ineq-JN23-(3-11)}
\abs{\zeta_F(Z_B)}
\leq\int_{Z'_B\cup Z''_B}{\rm J}^{Z_B}\zeta_F\rd\mcH^{n+1}
=&\int_{\S_B'}F(\nu)\int_0^{\tilde R}\prod_{i=1}^n\abs{1-t\kappa_i(x)}\rd t\rd\mcH^n\\
&+\int_{\S_B''}F(\nu)\int_0^{\frac1{\kappa_n(x)}}\prod_{i=1}^n\abs{1-t\kappa_i(x)}\rd t\rd\mcH^n.
}
Note that on $Z'_B$, one has by definition that $\abs{1-t\kappa_i(x)}=(1-t\kappa_i(x))$ for any $(x,t)\in\S_B'\times[0,\tilde R]$, and hence from the AM-GM inequality, the Jensen's inequality, and the definition of $\tilde R$ that: for any $(x,t)\in Z_B'$,
\eq{
\prod_{i=1}^n\abs{1-t\kappa_i(x)}
\leq(1-t\frac{H_\S(x)}n)^n
\leq(1+\frac{\tilde R}n\abs{H_\S(x)})^n
\leq C(n,\theta,C_0)(1+\abs{H_\S(x)}^n),
}
while similarly on $Z''_B$, since by definition $t\leq\frac1{\kappa_n(x)}<\tilde R$, one has
\eq{
\prod_{i=1}^n\abs{1-t\kappa_i(x)}
\leq C(n,\theta,C_0)(1+\abs{H_\S(x)}^n).
}
Taking these facts into account, \eqref{ineq-JN23-(3-11)} thus reads
\eq{
\abs{\zeta_F(Z_B)}
\leq& C(n,\theta,C_0)\tilde R\int_{\S_B}F(\nu)(1+\abs{H_\S(x)}^n)\rd\mcH^n\\
\leq& C(n,\theta,C_0)\int_{\S_B}(1+\lambda^n+\abs{H_\S-\lambda}^n)\rd\mcH^n\\
\leq&C(n,\theta,C_0)\left(\mcH^n(\S_B)+\norm{H_\S-\lambda}_{L^n(\S)}^n\right)
\leq C(n,\theta,C_0)\norm{H_\S-\lambda}_{L^n(\S)},
}
where we have used trivially $M_F=1+\abs{\cos\theta}$ in the second inequality, \eqref{ineq-JN23-Lem2.4-(i)} for the third inequality, and \eqref{ineq-JN23-(3-3)} for the last one.

\noindent{\bf Step 4. We prove that for any $s\geq0$, and for any $0<r<R$, there holds
\eq{\label{ineq-JN23-(3-13)}
\Abs{
\Om\cap\zeta_F\left(Z_G\cap(\S_s\times(r,R))\right)}
\geq\frac{(R-r)^{n+1}}{(n+1)R^n}\int_{\S_s}F(\nu)\rd\mcH^n-C(n,\theta,C_0)\norm{H_\S-\lambda}_{L^n(\S)}.
}
}

To prove \eqref{ineq-JN23-(3-13)}, we follow to use the idea presented in \cite{JN23} of "backtracking" the Montiel-Ros argument.

Precisely, invoking the definitions of $N_1,\ldots, N_4$, we may estimate with \eqref{ineq-JN23-(3-12)} as follows:
%for simplicity let us write $f_\theta:\S\ra(0,2),\text{ }x\mapsto(1-\cos\theta\nu(x)\cdot E_{n+1})$,
\eq{\label{esti-Z_G-Sigma_s-1}
&\Abs{
\Om\cap\zeta_F\left(Z_G\cap(\S_s\times(r,R))\right)}
\geq\abs{\zeta_F(Z_G\cap(\S_s\times(r,R)))}-N_1\\
\geq&\int_{\zeta_F(Z_G\cap(\S_s\times(r,R)))}\mcH^0(\zeta_F^{-1}(y)\cap Z_G\cap(\S_s\times(r,R)))\rd y-N_1-N_2\\
=&\int_{\S_G\cap\S_s}F(\nu)\int_{\min\{r,\frac1{\kappa_n(x)}\}}^{\min\{R,\frac1{\kappa_n(x)}\}}\prod_{i=1}^n\left(1-t\kappa_i(x)\right)\rd t\rd\mcH^n-N_1-N_2\\
\geq&\int_{\S_G\cap\S_s}F(\nu)\int_{\min\{r,\frac1{\kappa_n(x)}\}}^{\min\{R,\frac1{\kappa_n(x)}\}}\left(1-t\frac{H_\S(x)}n\right)^n\rd t\rd\mcH^n-N_1-N_2-N_3\\
\geq&\int_{\S_G\cap\S_s}F(\nu)\int_{\min\{r,\frac1{\kappa_n(x)}\}}^{\min\{R,\frac{n}{H_\S(x)}\}}\left(1-t\frac{H_\S(x)}n\right)^n\rd t\rd\mcH^n-N_1-N_2-N_3-N_4\\
\geq&\int_{\S_G\cap\S_s}F(\nu)\int_{\min\{r,\frac{n}{H_\S(x)}\}}^{\min\{R,\frac{n}{H_\S(x)}\}}\left(1-t\frac{H_\S(x)}n\right)^n\rd t\rd\mcH^n-N_1-N_2-N_3-N_4,
}
where we have used the fact that $\frac1{\kappa_n(x)}\leq\frac{n}{H_\S(x)}$ on $\S_G$ to derive the last inequality.
Let us investigate further the integral arises in the last inequality, recall that on the "good" set $\S_G$, we have $0<\frac12\lambda\leq H_\S(x)\leq2\lambda$,
%also notice that for $0<t<R=\frac{n}\lambda$,
%\eq{
%0
%<(1-t\frac\lambda n)^n
%\leq 2^{n-1}\left(1-t\frac{H_\S}n\right)^n+2^{n-1}(\frac{t}n)^n(H_\S-\lambda)^n,
%}
therefore we find
\eq{\label{esti-Z_G-Sigma_s-2}
&\int_{\S_G\cap\S_s}F(\nu)\int_{\min\{r,\frac{n}{H_\S(x)}\}}^{\min\{R,\frac{n}{H_\S(x)}\}}\left(1-t\frac{H_\S(x)}n\right)^n\rd t\rd\mcH^n\\
\geq&\int_{\S_G\cap\S_s}F(\nu)\int_{\min\{r,\frac{n}{H_\S(x)}\}}^{\min\{R,\frac{n}{H_\S(x)}\}}\left(1-t\frac\lambda n\right)^n\rd t\rd\mcH^n-C(n,\theta,C_0)\norm{H_\S-\lambda}_{L^n(\S)}\\
\geq&\int_{\S_G\cap\S_s}F(\nu)\int_{r}^{R}\left(1-t\frac\lambda n\right)^n\rd t\rd\mcH^n-C(n,\theta,C_0)\norm{H_\S-\lambda}_{L^n(\S)}\\
=&\frac{(R-r)^{n+1}}{(n+1)R^n}\int_{\S_G\cap\S_s}F(\nu)\rd\mcH^n-C(n,\theta,C_0)\norm{H_\S-\lambda}_{L^n(\S)},
}
to derive the second inequality, we have used first the fact that
\eq{
\int_{\min\{r,\frac{n}{H_\S(x)}\}}^{\min\{R,\frac{n}{H_\S(x)}\}}\left(1-t\frac\lambda n\right)^n\rd t
\geq\int_{r}^{\min\{R,\frac{n}{H_\S(x)}\}}\left(1-t\frac\lambda n\right)^n\rd t,
}
and then the observation: as $\frac{n}\lambda=R>\frac{n}{H_\S(x)}$, one has
\eq{
\int_{\frac{n}{H_\S(x)}}^R(1-t\frac\lambda{n})^n\rd t
=\frac{n}{(n+1)\lambda}(\frac{H_\S(x)-\lambda}{H_\S(x)})^{n+1}
<\frac{n}{n+1}(H_\S(x)-\lambda)^n(\frac1\lambda)^{n+1},
}
it follows from $\de<1$ that
\eq{
\int_{\S_G\cap\S_s\cap\{H_\S(x)>\lambda\}}F(\nu)\int_{\frac{n}{H_\S(x)}}^R(1-t\frac\lambda{n})^n\rd t
\leq& C(n,\theta,C_0)\int_\S\abs{H_\S-\lambda}^n\\
\leq&C(n,\theta,C_0)\norm{H_\S-\lambda}_{L^n(\S)}.
}

Substituting \eqref{esti-Z_G-Sigma_s-2} back into \eqref{esti-Z_G-Sigma_s-1}, and keeping in mind that the error terms are controlled \eqref{ineq-JN23-(3-12)}, we thus arrive at
\eq{
&\Abs{
\Om\cap\zeta_F\left(Z_G\cap(\S_s\times(r,R))\right)}\\
\geq&\frac{(R-r)^{n+1}}{(n+1)R^n}\int_{\S_G\cap\S_s}F(\nu)\rd\mcH^n-C(n,\theta,C_0)\norm{H_\S-\lambda}_{L^n(\S)}.
}
Finally, by virtue of \eqref{ineq-JN23-(3-3)}, we have
\eq{
&\int_{\S_s}F(\nu)\rd\mcH^n
=\int_{\S_G\cap\S_s}F(\nu)\rd\mcH^n+\int_{\S_B\cap\S_s}F(\nu)\rd\mcH^n\\
\leq&\int_{\S_G\cap\S_s}F(\nu)\rd\mcH^n+M_F\mcH^n(\S_B)
\leq\int_{\S_G\cap\S_s}F(\nu)\rd\mcH^n+C(n,\theta,C_0)\norm{H_\S-\lambda}_{L^n(\S)},
}
and hence
\eq{
&\frac{(R-r)^{n+1}}{(n+1)R^n}\int_{\S_G\cap\S_s}F(\nu)\rd\mcH^n\\
\geq&\frac{(R-r)^{n+1}}{(n+1)R^n}\int_{\S_s}F(\nu)\rd\mcH^n-\frac{n}{(n+1)\lambda}C(n,\theta,C_0)\norm{H_\S-\lambda}_{L^n(\S)},
}
from which we deduce \eqref{ineq-JN23-(3-13)}.

\noindent{\bf Step 5. We prove two inclusions: for any $0<\rho<r<R$,
\eq{\label{ineq-JN23-(3-15)}
&\Om\cap\zeta_F\left(Z_G\cap(\S_0\times(r,R))\right)\\
\subseteq&\Om_r\cup\left\{y\in\zeta_F(Z_G):\mcH^0(\zeta_F^{-1}(y)\cap Z_G)\geq2\right\}\cup\zeta_F(Z_B),
}
and
\eq{\label{ineq-JN23-(3-19)}
&\Om\cap\zeta_F(Z_G\cap(\S_{r}\times(r-\rho,R)))\\
\subseteq&(\Om_r+\mcW_{\rho})\cap\overline{\mfR^{n+1}_+}\cup\left\{y\in\zeta_F(Z_G):\mcH^0(\zeta_F^{-1}(y)\cap Z_G)\geq2\right\}\cup\zeta_F(Z_B).
}
%which are direct consequences of the triangle inequality.
}

In fact, for any $y\in\Om\cap\zeta_F\left(Z_G\cap(\S_0\times(r,R))\right)$, there exist $x_y\in\S_G$, $r<t_y<R$ with $t_y\leq\frac1{\kappa_n(x_y)}$, such that $y=\zeta_F(x_y,t_y)=x_y-t_y\nu_F(x_y)$.

If $u_\theta(y)$ attains at $x_y$, then since $y\in\Om$, we learn from \cite{JWXZ23}*{Proof of Theorem 1.2} that $x_y$ must lie in the interior of $\S$, the ball $\overline{\mcW}_{t_y}(y)$ is tangent to $\S$ at $x_y$ and touches $\S$ from the interior.
It follows that $u_\theta(y)=t_y>r$, so that $y\in\Om_r$.

If $u_\theta(y)$ attains at some $x_y'\neq x_y$, then we must have $u_\theta(y)<t_y$, and again $x_y'\notin\p\S$, thus we may write $y
=x_y'-u_\theta(y)\nu_F(x_y')
=\zeta_F(x_y',u_\theta(y))$, from which we infer easily that: if $(x_y',u_\theta(y))\in Z_B$, then one has $y\in\zeta_F(Z_B)$; while if $(x_y',u_\theta(y))\notin Z_B$, then we must have  $x_y'\in\S_G$.
On the other hand, since the ball $\overline{\mcW}_{u_\theta(y)}(y)$ is tangent to $\S$ at $x_y'$ and touches $\S$ from the interior, there holds $u_\theta(y)\leq\frac1{\kappa_n(x_y')}$,
thereby $(x_y',u_\theta(y))\in Z_G$, and \eqref{ineq-JN23-(3-15)} follows since
\eq{
y
=\zeta_F(x_y,t_y)
=\zeta_F(x_y',u_\theta(y)).
}

To show \eqref{ineq-JN23-(3-19)}, we consider any $y\in\Om\cap\zeta_F(Z_G\cap(\S_{r}\times(r-\rho,R)))$, i.e., there exist $x_y\in\S_G\cap\S_{r}$ and $t_y\in(r-\rho,R)$ with $t_y\leq\frac1{\kappa_n(x_y)}$, such that $y=\zeta_F(x_y,t_y)$.
As before, $x_y\notin\p\S$.

If $t_y\in(r,R)$, since  $\S_{r}\subset\S=\S_0$, we may exploit \eqref{ineq-JN23-(3-15)} directly to find
\eq{
y\in\Om_r\cup\left\{y\in\zeta_F(Z_G):\mcH^0(\zeta_F^{-1}(y)\cap Z_G)\geq2\right\}\cup\zeta_F(Z_B).
}

If $t_y\in(r-\rho,r]$,
we may then write \eq{
y=x_y-r\nu_F(x_y)+(r-t_y)\nu_F(x_y).
}
Since $x_y\in\S_r$, we see that $x_y-r\nu_F(x)$ belongs to $\Om_r$.
On the other hand, since $r-t_y\in[0,\rho)$, we must have $(r-t_y)\nu_F(x_y)$ belongs to $\mcW_{\rho}$.
\eqref{ineq-JN23-(3-19)} follows
easily.

\noindent{\bf Step 6. We prove \eqref{ineq-Om_r} and \eqref{ineq-Sigmma-Sigma_r}.
}

Recall that by \eqref{ineq-JN23-(3-12)}, \eqref{ineq-JN23-(3-10)}, $\Abs{\left\{y\in\zeta_F(Z_G):\mcH^0(\zeta_\theta^{-1}(y)\cap Z_G)\geq2\right\}\cup\zeta_F(Z_B)}$ is controlled by $\norm{H_\S-\lambda}_{L^n(\S)}$, and hence we may exploit the inclusion \eqref{ineq-JN23-(3-15)} (note that $\S_0=\S$), in conjunction with the estimate \eqref{ineq-JN23-(3-13)}, to obtain
\eq{\label{ineq-JN23-(3-14)}
\abs{\Om_r}
\geq&\Abs{
\Om\cap\zeta_F\left(Z_G\cap(\S_0\times(r,R))\right)}-C(n,\theta,C_0)\norm{H_\S-\lambda}_{L^n(\S)}\\
\geq&\frac{(R-r)^{n+1}}{(n+1)R^n}\int_{\S}F(\nu)\rd\mcH^n-C(n,\theta,C_0)\norm{H_\S-\lambda}_{L^n(\S)}\\
=&\frac{P_F(\Om;\mfR^{n+1}_+)}{(n+1)R^n}(R-r)^{n+1}-C(n,\theta,C_0)\norm{H_\S-\lambda}_{L^n(\S)}\\
\overset{\eqref{ineq-JN23-(3-4)}}{\geq}&\frac{\abs{\Om}}{R^{n+1}}(R-r)^{n+1}-C(n,\theta,C_0)\norm{H_\S-\lambda}_{L^n(\S)}.
}
It is left to prove the other direction of
\eqref{ineq-Om_r}, for clarification we separate the proof into the following claims.

{\bf Claim 1.} The refined version of \eqref{ineq-Om-zeta_theta} holds, precisely, for any $0\leq s<t$, there holds
\eq{\label{ineq-Om-zeta_theta-refined}
\Om_s\setminus\Om_t
\subset\zeta_F(Z_G^{s,t})\cup\zeta_F(Z_B),
}
where $Z_G^{s,t}=Z_G\cap(\S_s\times[s,t])$.

To see this, let us fix any $y\in\Om_s\setminus\Om_t$, by definition we shall have $s<u_\theta(y)\leq t$, and hence there exists $x_y\in\S_{u_\theta(y)}\subset\S_s$, at which $u_\theta(y)$ is attained.
Clearly, we must have $u_\theta(y)\leq\frac1{\kappa_n(x_y)}$ if $\kappa_n(x_y)>0$.

If $x_y\in\S_G$, then it is easy to see that $(x_y,u_\theta(y))\in Z_G^{s,t}$, and hence $y=\zeta_F(x_y,u_\theta(y))\in\zeta_F( Z_G^{s,t})$.

If $x_y\in\S_B'$, since by definition $\tilde R\geq u_\theta(y)$, we have $(x_y,u_\theta(y))\in Z_B'$, so that $y\in\zeta_F(Z_B')$;
if $x_y\in\S_B''$, we see that $\kappa_n(x_y)>0$ by definition, and hence $u_\theta(y)\leq\frac1{\kappa_n(x_y)}$ as argued above, showing that $y\in\zeta_F(Z_B'')$. In particular, this proves \eqref{ineq-Om-zeta_theta-refined}.

{\bf Claim 2.} $\abs{\Om_R}$ is almost negligible in terms of the $L^n$-deficit.

We first observe that in the statement together with the proof of \eqref{ineq-Om-zeta_theta-refined}, if we take $s=R$ and $t=\infty$, we shall get
\eq{\label{ineq-Om_R-zeta_theta}
\Om_R\subset\zeta_F(Z_G^{R,\infty})\cup\zeta_F(Z_B).
}
Notice also that, on $\S_G$ one has $\frac12\lambda\leq H_\S\leq2\lambda$, thus if in addition $R=\frac{n}\lambda<\frac{n}{H_\S(x)}$, then $0<\frac{\lambda-H_\S(x)}{H_\S(x)}<1$.
Taking also \eqref{ineq-JN23-(3-10)} into account, we may use the inclusion \eqref{ineq-Om_R-zeta_theta} to find
\eq{
&\abs{\Om_R}
\leq\abs{\zeta_F(Z_G^{R,\infty})}+\abs{\zeta_F(Z_B)}\\
\leq&\int_{Z_G}F(\nu)\int_R^{\max\{R,\frac{n}{H_\S(x)}\}}(1-t\frac{H_\S(x)}n)^n\rd t\rd\mcH^n+C(n,\theta,C_0)\norm{H_\S-\lambda}_{L^n(\S)}\\
\leq&C(n)\int_{Z_G}F(\nu)\lambda^{-(n+1)}(\lambda-H_\S)^n\rd\mcH^n+C(n,\theta,C_0)\norm{H_\S-\lambda}_{L^n(\S)}\\
\leq& C(n,\theta,C_0)\norm{H_\S-\lambda}_{L^n(\S)},
}
which proves the claim.

Let us finish the proof of \eqref{ineq-Om_r}, by using {\bf Claim 2}, then {\bf Claim 1} (with $s=r$, $t=R=\frac{n}{\lambda}$), and also \eqref{ineq-JN23-(3-10)}, we find
\eq{\label{ineq-JN23-(3-17)}
\abs{\Om_r}
\leq&\Abs{\Om_r\setminus\Om_R}+C(n,\theta,C_0)\norm{H_\S-\lambda}_{L^n(\S)}\\
\leq&\abs{\zeta_F(Z_G^{r,R})}+\abs{\zeta_F(Z_B)}+C(n,\theta,C_0)\norm{H_\S-\lambda}_{L^n(\S)}\\
\leq&\int_{\S_G\cap\S_r}F(\nu)\int_r^R(1-t\frac{H_\S(x)}n)^n\rd t\rd\mcH^n+C(n,\theta,C_0)\norm{H_\S-\lambda}_{L^n(\S)}\\
\leq&\int_{\S_G\cap\S_r}F(\nu)\int_r^R(1-t\frac\lambda{n})^n\rd t\rd\mcH^n+C(n,\theta,C_0)\norm{H_\S-\lambda}_{L^n(\S)}\\
\leq&\frac{(R-r)^{n+1}}{(n+1)R^n}\int_{\S_r} F(\nu)\rd\mcH^n+C(n,\theta,C_0)\norm{H_\S-\lambda}_{L^n(\S)},%\\
%\overset{\eqref{ineq-JXZ23-(20)}}{=}&\frac{\abs{\S}-\cos\theta\abs{T}}{(n+1)R^n}(R-r)^{n+1}+C(n,d_m,\theta,C_0)\norm{H_\S-\lambda}_{L^n(\S)},
}
after using the fact that $\S_r\subset\S$, %\eqref{ineq-JXZ23-(20)},
then \eqref{ineq-JN23-(3-4)}, we arrive at
\eq{\label{ineq-JN23-(3-17)'}
\abs{\Om_r}
\leq\frac{\abs{\Om}}{R^{n+1}}(R-r)^{n+1}+C(n,\theta,C_0)\norm{H_\S-\lambda}_{L^n(\S)}.
}

\eqref{ineq-Sigmma-Sigma_r} is a direct consequence of the combination of \eqref{ineq-JN23-(3-17)} and the second inequality in \eqref{ineq-JN23-(3-14)}.

\noindent{\bf Step 7. We complete the proof by showing \eqref{ineq-Om_r-rho}.
}

{\bf Claim 3.} $(\Om_r+\mcW_{\rho})\cap\overline{\mfR^{n+1}_+}\subset\Om_{r-\rho}$.

To see this, we fix any $\tilde y\in(\Om_r+\mcW_{\rho})\cap\overline{\mfR^{n+1}_+}$, which can be decomposed to be
\eq{
\tilde y=y+\xi,\text{ }y\in\Om_r,\text{ }\xi\in \mcW_{\rho}.
}
For any $z\in\S$, we may use the triangle inequality for $F^o$ to find
\eq{
F^o(z-\tilde y)+F^o(\xi)
=F^o(z-y-\xi)+F^o(\xi)
\geq F^o(z-y)\geq u_\theta(y)
>r,
}
so that
\eq{
F^o(z-\tilde y)>r-F^o(\xi)>r-\rho\text{ for every }z\in\S,
}
implying that $\tilde y\in\Om_{r-\rho}$, and proves the claim.

Firstly, we exploit the inclusion \eqref{ineq-JN23-(3-19)}, in conjunction with the estimates \eqref{ineq-JN23-(3-12)}, \eqref{ineq-JN23-(3-10)}, \eqref{ineq-JN23-(3-13)} (letting $s=r$), \eqref{ineq-Sigmma-Sigma_r}, and then \eqref{ineq-JN23-(3-4)} (recall that $\lambda=\frac{n}R$) to get
\eq{
&\Abs{(\Om_r+\mcW_{\rho})\cap\overline{\mfR^{n+1}_+}}\\
\geq&\Abs{
\Om\cap\zeta_F\left(Z_G\cap(\S_r\times(r-\rho,R))\right)}-C(n,\theta,C_0)\norm{H_\S-\lambda}_{L^n(\S)}\\
\geq&\frac{(R-(r-\rho))^{n+1}}{(n+1)R^n}\int_{\S_r}F(\nu)\rd\mcH^n-C(n,\theta,C_0)\norm{H_\S-\lambda}_{L^n(\S)}\\
\geq&\frac{(R-(r-\rho))^{n+1}}{(n+1)R^n}\int_{\S}F(\nu)\rd\mcH^n-\frac{C(n,\theta,C_0)}{(R-r)^{n+1}}\norm{H_\S-\lambda}_{L^n(\S)}\\
=&\frac{P_F(\Om;\mfR^{n+1}_+)}{(n+1)R^n}(R-(r-\rho))^{n+1}-\frac{C(n,\theta,C_0)}{(R-r)^{n+1}}\norm{H_\S-\lambda}_{L^n(\S)}\\
\geq&\frac{\abs{\Om}}{R^{n+1}}(R-(r-\rho))^{n+1}-\frac{C(n,\theta,C_0)}{(R-r)^{n+1}}\norm{H_\S-\lambda}_{L^n(\S)}.
}
On the other hand, exploiting {\bf Claim 3}, and \eqref{ineq-JN23-(3-17)'} with $r$ replaced by $r-\rho$, we find
\eq{
\Abs{(\Om_r+\mcW_{\rho})\cap\overline{\mfR^{n+1}_+}}
\leq\abs{\Om_{r-\rho}}
\leq\frac{\abs{\Om}}{R^{n+1}}(R-(r-\rho))^{n+1}+C(n,\theta,C_0)\norm{H_\S-\lambda}_{L^n(\S)},
}
which completes the proof.
\end{proof}

\subsection{Quantitative capillary Alexandrov theorem}

We have now all the requisites to prove our main theorem.

\begin{proof}[Proof of Theorem \ref{Thm-Stability}]
We begin with the notification that, in the proof $C=C(n,\theta,C_0)$ shall be used to denote positive constants that depend only on the dimension $n$, the prescribed capillary angle $\theta$, and the isoperimetric control $C_0$.
The values of $C(n,\theta,C_0)$ may vary from line to line, and the shorthand $C$ shall be adopted unless there is any possible confusion.

For simplicity we denote the $L^n$-deficit as
\eq{
\ep\coloneqq\norm{H_\S-\lambda}_{L^n(\S)}.
}

If $\ep=0$, then we know that $H_\S=\lambda$ for $\mcH^n$-a.e. $x\in\S$, and hence \eqref{formu-1st-variation} can be written as
\eq{
\int_\S{\rm div}_\S X\rd\mcH^n-\cos\theta\int_T{\rm div}_{\p\mbR^{n+1}_+}X\rd\mcH^n
=\lambda\int_\S  X\cdot\nu_\S\rd\mcH^n,
}
for any $X\in C_c^1(\mfR^{n+1},\mfR^{n+1})$ tangent to $\p\mfR^{n+1}_+$;
in other words, $\Om$ is stationary for the $\mcA$-functional, defined in \cite{XZ23}*{Definition 1.1}, so that from \cite{XZ23}*{Theorem 1.3} we deduce that $\Om$ is a disjoint union of $\theta$-balls.

\noindent{\bf Preliminary step: set-ups.}

Let us now continue with $0<\ep\leq\de$, where $\de$ is firstly taken from Proposition \ref{Prop-JN23-Prop3.3}.
We shall choose $\de$ to be further small (if needed) in due course, with the choices depending only on $n,\theta,C_0$.
With this initial choice, we immediately learn from
Proposition \ref{Prop-JN23-Lem2.4} and Remark \ref{Rem-apriori} that there exist positive constants $C$, such that
\eq{
\frac1C
\leq\lambda\leq C,\quad
\frac1C\leq R=\frac{n}\lambda\leq C,
}
and the number of connected components of $\Om$ and their diameters are bounded by some $C$ as well.
We also note that thanks to \eqref{ineq-JN23-(3-4)}, $\abs{\Om}$ can be bounded from above by some $C$, provided that $\de$ is chosen small enough.

We will always assume that $\ep\leq\de<1$, which implies the following relations:
\eq{
\ep
<\ep^\frac{n+1}{n+2}
<\ep^\frac1{n+2}
\leq\ep^\frac1{n(n+2)}
<\ep^\frac1{(n+2)^2}
<1.
}

Let us first decrease $\de$, if necessary, so that $R-\de^\frac{1}{n+2}>\frac12R$, and write
\eq{\label{defn-r_0}
r_0
\coloneqq R-\ep^\frac{1}{n+2}>\frac12R.
}
Thanks to our choice of $\de$, we may rewrite the estimates in Proposition \ref{Prop-JN23-Prop3.3} as
\eq{\label{ineq-JN23-3-21}
\Abs{\abs{\Om_r}-\frac{\abs{\Om}}{R^{n+1}}(R-r)^{n+1}}
\leq C\ep,
}
for any $0<r<R$; and
\eq{\label{ineq-JN23-3-22}
\Abs{\Abs{(\Om_r+\mcW_{\rho})\cap\overline{\mfR^{n+1}_+}}-\frac{\abs{\Om}}{R^{n+1}}(R-(r-\rho))^{n+1}}
\leq\frac{C}{(R-r_0)^{n+1}}\ep
\leq C\ep^\frac{1}{n+2},
}
for any $0\leq\rho\leq r\leq r_0$.

Using \eqref{ineq-JN23-3-21} with $r=r_0$, we find
\eq{
\abs{\Om_{r_0}}
\geq\frac{\abs{\Om}}{R^{n+1}}\ep^\frac{n+1}{n+2}-C\ep
\geq\frac1C\ep^\frac{n+1}{n+2}-C\ep,
}
and hence $\Om_{r_0}$ is nonempty after possibly decreasing $\de$.
Therefore, for any $r'>r_0$ with $r'-r_0$ small enough, we shall have that $\Om_{r'}$ is non-empty as well.
Moreover, for any $x\in\S_{r'}$, by definition we could find some $y_x\in\overline\Om$ such that $u_\theta(y_x)=r_0$ and attains at $x$, that is, we could write $x=y_x+r_0\nu_F(x)$,
meaning that $\S_{r'}\subset(\pr\Om_{r_0}+\overline{\mcW_{r_0}})\cap\overline{\mfR^{n+1}_+}$.
Since $r_0=R-\ep^\frac1{n+2}$, we can then conclude from \eqref{ineq-Sigmma-Sigma_r} that
\eq{
&\mcH^n(\S\setminus(\overline{\Om}_{r_0}+\overline{\mcW_{r_0}})\cap\overline{\mfR^{n+1}_+})
\leq \mcH^n(\S\setminus\S_{r'})\\
\leq&\frac1{m_F}\int_{\S\setminus\S_{r'}}F(\nu)\rd\mcH^n
\leq C\frac\ep{\left(r_0-r'+\ep^\frac1{n+2}\right)^{n+1}}.
}
Letting $r'\ra r_0^+$, this reads
\eq{\label{ineq-JN23-(3-23)}
\mcH^n(\S\setminus(\overline{\Om}_{r_0}+\overline{\mcW_{r_0}})\cap\overline{\mfR^{n+1}_+})
\leq C\ep^\frac1{n+2}.
}

\noindent{\bf Step 1. We prove that there exists a positive constant
\eq{\label{defn-D_0}
D_0=D_0(n,\theta,C_0)\leq\frac{1-\abs{\cos\theta}}4R,
}
such that for any $x,y\in\Om_{r_0}$,
\eq{\label{ineq-JN23-(3-24)}
\text{either }\abs{x-y}<(1-\abs{\cos\theta})\ep^\frac1{2(n+2)},
\text{ or }
\abs{x-y}\geq D_0.
}
}

Let us write $D\coloneqq\abs{x-y}$ and denote the geodesic segment joining $y,x$ by
\eq{
\underline{xy}
\coloneqq\{tx+(1-t)y:t\in[0,1]\}.
}
We shall assume that 
\eq{\label{ineq-JN23-(3-25)}
D\leq\min\{\frac{1-\abs{\cos\theta}}4R,1\},
}
otherwise \eqref{ineq-JN23-(3-24)} trivially holds.
It follows from \eqref{defn-r_0} that $r_0-\frac{D^2}{(1-\abs{\cos\theta})^2)R}>0$, and hence $\Om_{r_0-\frac{D^2}{(1-\abs{\cos\theta})^2R}}$ is well-defined and nonempty because $\Om_{r_0}$ is nonempty.

We claim that
\eq{
\underline{xy}\subset\Om_{r_0-\frac{D^2}{(1-\abs{\cos\theta)}^2R}}.
}
To see this, let us set $z'$ to be the point on $\underline{xy}$ such that
\eq{
u_\theta(z')
=\min_{p\in\underline{xy}}u_\theta(p),
}
and let $z\in\S$ be the point on which $u_\theta(z')$ is attained.

If $z'=x$ or $y$, then it is easy to see that $\underline{xy}\subset\Om_{r_0}$.
In the case that $z'\neq x,y$,
without loss of generality, we assume that $\abs{x-z'}\leq\frac12\abs{x-y}=\frac12D$.
Clearly,
we shall have $u_\theta(z')\leq u_\theta(x)$.
We may suppose that $u_\theta(z')<r_0$, otherwise we again simply have $\underline{xy}\subset\Om_{r_0}$.

Let us consider the geodesic segment $\gamma$ joining the points 
\eq{
(x-u_\theta(z')\cos\theta E_{n+1}), (y-u_\theta(z')\cos\theta E_{n+1}),
}
which apparently parallels to $\underline{xy}$.
Note also that $z'-u_\theta(z')\cos\theta E_{n+1}\in{\rm int}(\gamma)$.

From the definitions of the shifted distance function and the point $z'$, we know that $\p B_{u_\theta(z')}(z)$ and $\gamma$ are mutually tangent at $z'$, that is to say, $z-(z'-u_\theta(z')\cos\theta E_{n+1})$ is orthogonal to $\gamma$.
On the other hand, a simple application of the triangle inequality gives
\eq{
&\Abs{x-u_\theta(z')\cos\theta E_{n+1}-z}
=\Abs{x-r_0\cos\theta E_{n+1}-z-(u_\theta(z')-r_0)\cos\theta E_{n+1}}\\
\geq&r_0-(r_0-u_\theta(z'))\abs{\cos\theta}
=r_0(1-\abs{\cos\theta})+u_\theta(z')\abs{\cos\theta},
}
where we have used in the first inequality the fact that $r_0<u_\theta(x)$ so that 
\eq{
\abs{x-r_0\cos\theta E_{n+1}-z}\geq u(x-r_0\cos\theta E_{n+1})>r_0
}
thanks to \eqref{ineq-weighted-distance}.
By virtue of the Pythagorean theorem, we obtain
\eq{
&\Abs{x-u_\theta(z')\cos\theta E_{n+1}-z}^2\\
=&\Abs{x-u_\theta(z')\cos\theta E_{n+1}-z'+u_\theta(z')\cos\theta E_{n+1}}^2+\Abs{z'-u_\theta(z')\cos\theta E_{n+1}-z}^2,
}
combining with the previous observations, we get
\eq{
\left(r_0(1-\abs{\cos\theta})+u_\theta(z')\abs{\cos\theta}\right)^2
\leq\frac14D^2+u_\theta(z')^2.
}
Expanding the above expression yields
\eq{
r_0^2(1-\abs{\cos\theta})^2+2r_0u_\theta(z')\abs{\cos\theta}(1-\abs{\cos\theta})
\leq\frac14D^2+u_\theta(z')^2(1-\abs{\cos\theta}^2),
}
shrinking the first order term by virtue of $r_0>u_\theta(z')$, we may rearrange this to read
\eq{
r_0^2(1-\abs{\cos\theta})^2
\leq\frac14D^2+u_\theta(z')^2(1-\abs{\cos\theta})^2.
}
On the other hand, by virtue of \eqref{defn-r_0} and \eqref{ineq-JN23-(3-25)}, we may use a direct computation to find
\eq{
\left(r_0^2-\frac{D^2}{4(1-\abs{\cos\theta})^2}\right)^\frac12
\geq r_0-\frac{D^2}{(1-\abs{\cos\theta})^2R},
}
therefore implying that
\eq{
u_\theta(z')
\geq r_0-\frac{D^2}{(1-\abs{\cos\theta})^2R},
}
which proves the claim.

To proceed,
recalling that ${\rm Lip}(u_\theta)\leq\frac1{1-\abs{\cos\theta}}$, it is then easy to find (if $\rho$ is such that $r-\frac{\rho^2}{(1-\abs{\cos\theta})}>0$)
\eq{
\left(\Om_r+B_{\rho^2}\right)\cap\mfR^{n+1}_+
\subset\Om_{r-\frac{\rho^2}{(1-\abs{\cos\theta})}},
}
and hence
\eq{
\left(\Om_{r_0-\frac{D^2}{(1-\abs{\cos\theta)}^2R}}+B_{D^2}\right)\cap\mfR^{n+1}_+
\subset\Om_{r_0-(1-\abs{\cos\theta}+\frac1R)\frac{D^2}{(1-\abs{\cos\theta})^2}},
}
where the super level-set $\Om_{r_0-(1-\abs{\cos\theta}+\frac1R)\frac{D^2}{(1-\abs{\cos\theta})^2}}$ is well-defined and nonempty thanks to \eqref{defn-r_0} and \eqref{ineq-JN23-(3-25)}.
More precisely,
\eq{
r_0-(1-\abs{\cos\theta}+\frac1R)\frac{D^2}{(1-\abs{\cos\theta})^2}
>&r_0-\frac{D}{1-\abs{\cos\theta}}-\frac1R\frac{D^2}{(1-\abs{\cos\theta})^2}\\
>&\frac12R-\frac14R-\frac1{16}R>0.
}

Notice that $x,y\in\overline{\mfR^{n+1}_+}$, therefore at least half of the round cylinder $\underline{xy}\times B^n_{D^2}$ is contained in
\eq{
\left(\underline{xy}+B_{D^2}\right)\cap\overline{\mfR^{n+1}_+}\subset\left(\Om_{r_0-\frac{D^2}{(1-\abs{\cos\theta)}^2R}}+B_{D^2}\right)\cap\mfR^{n+1}_+
\subset\Om_{r_0-(1-\abs{\cos\theta}+\frac1R)\frac{D^2}{(1-\abs{\cos\theta})^2}}.
}
Exploiting \eqref{ineq-JN23-3-21}, we thus arrive at
\eq{
\frac12\om_nD^{1+2n}
\leq&\Abs{\Om_{r_0-(1-\abs{\cos\theta}+\frac1R)\frac{D^2}{(1-\abs{\cos\theta})^2}}}
\leq\frac{\abs{\Om}}{R^{n+1}}\left(R-r_0+\frac{1-\abs{\cos\theta}+\frac1R}{(1-\abs{\cos\theta})^2}D^2\right)^{n+1}+C\ep\\
=&\frac{\abs{\Om}}{R^{n+1}}\left(\ep^\frac1{n+2}+\frac{1-\abs{\cos\theta}+\frac1R}{(1-\abs{\cos\theta})^2}D^2\right)^{n+1}+C\ep
\leq C\ep^\frac{n+1}{n+2}+CD^{2(n+1)}.
}
Therefore, whether $D\leq(1-\abs{\cos\theta})\ep^\frac1{2(n+2)}$, or $D\geq(1-\abs{\cos\theta})\ep^\frac1{2(n+2)}$ so that
\eq{
\ep^\frac{n+1}{n+2}
\leq C(n,\theta)D^{2(n+1)},
}
and hence the above estimate reads
\eq{
\frac12\om_n D^{1+2n}
\leq CD^{2(n+1)},
}
implying that $D\geq C(n,\theta,C_0)>0$.
\eqref{ineq-JN23-(3-24)} then follows from the assumption \eqref{ineq-JN23-(3-25)}.

In view of \eqref{defn-m-M-F^o}, we may rewrite \eqref{ineq-JN23-(3-24)} in the following way: for any $x,y\in\Om_{r_0}$,
\eq{\label{ineq-JN23-(3-24)-F^o}
\text{either }F^o(x-y)<\ep^\frac1{2(n+2)},
\text{ or }F^o(x-y)\geq\frac{D_0}{1+\abs{\cos\theta}}\eqqcolon D_1.
}
By virtue of this estimate and after further decreasing $\de$, if needed, so that \eq{\label{ineq-delta-D_1}
\ep^\frac1{2(n+2)}
\leq\de^\frac1{2(n+2)}<\frac{1-\abs{\cos\theta}}8D_1,
}
we may decompose $\Om_{r_0}$ into $N$ clusters $\Om_{r_0}^1,\ldots,\Om_{r_0}^N$ by fixing a point $o_i\in\Om_{r_0}$ and defining $\Om_{r_0}^i$ as
\eq{
\Om_{r_0}^i
\coloneqq\{x\in\Om_{r_0}:F^o(x-o_i)
\leq\frac{1-\abs{\cos\theta}}8D_1\},
}
since for any $x,y\in\Om^i_{r_0}$, we have
\eq{
F^o(x-y)
\leq&F^o(x-o_i)+F^o(o_i-y)
\leq F^o(x-o_i)+\frac{M_{F^o}}{m_{F^o}}F^o(y-o_i)\\
\overset{\eqref{defn-m-M-F^o}}{\leq}&\frac{1-\abs{\cos\theta}}8D_1+\frac{1+\abs{\cos\theta}}8D_1
<\frac18D_1+\frac28D_1<D_1.
}

For the sake of simplicity, we denote by $\ep_0=\ep^\frac1{2(n+2)}$.
Clearly, from \eqref{ineq-JN23-(3-24)-F^o}, we shall have $\Om_{r_0}^i\subset \mcW_{\ep_0}(o_i)\cap\overline{\mfR^{n+1}_+}$, with $F^o(o_i-o_j)\geq D_1$ whenever $i\neq j$.
Thus, using the triangle inequality for $F^o$, we find: for every $\rho>0$,
\eq{\label{ineq-JN23-(3-28)}
\bigcup_{i=1}^N\left(\mcW_{\rho}(o_i)\cap\overline{\mfR^{n+1}_+}\right)
\subset(\Om_{r_0}+\mcW_{\rho})\cap\overline{\mfR^{n+1}_+}
\subset\bigcup_{i=1}^N\left(\mcW_{\rho+\ep_0}(o_i)\cap\overline{\mfR^{n+1}_+}\right).
}

Then we set $\tilde D_1=(1-\abs{\cos\theta})D_1$ and choose $\rho=\frac14\tilde D_1$,
thanks to \eqref{ineq-JN23-(3-24)-F^o}, \eqref{ineq-delta-D_1}, we see that $\mcW_{\frac14\tilde D_1}(o_i)\cap\mfR^{n+1}_+$ are mutually disjoint.
More precisely, for any $y\in\mcW_{\frac14\tilde D_1}(o_j)\cap\mfR^{n+1}_+$, we have for any $i\neq j$,
\eq{
D_1
\leq F^o(o_j-o_i)
\leq&F^o(y-o_i)+F^o(o_j-y),
}
and
\eq{
F^o(o_j-y)
\leq\frac{M_{F^o}}{m_{F^o}}F^o(y-o_j)
\leq\frac{1+\abs{\cos\theta}}4D_1
<\frac12D_1,
}
so that $y\notin\mcW_{\frac14\tilde D_1}(o_i)\cap\mfR^{n+1}_+$, since
\eq{
F^o(y-o_i)
\geq D_1-F^o(o_j-y)
>\frac12D_1>\frac14\tilde D_1.
}

Also, because $D_1\leq D_0$, we thus deduce from \eqref{defn-r_0} and \eqref{defn-D_0} that each $\mcW_{\frac14D_1}(o_i)\cap\mfR^{n+1}_+$ is contained in $\Om$, which in turn implies that the number of clusters $N$ is bounded from above by some positive constant $N_0=N_0(n,\theta,C_0)$.

To proceed, we denote by $S_i=\p \mcW_{\frac14\tilde D_1}(o_i)\cap\overline{\mfR^{n+1}_+}$ the spherical caps supported on $\p\mfR^{n+1}_+$, and $\theta_i$ the corresponding contact angle.
Clearly, it must be that $\theta_i\in[\theta,\pi]$; in the case that $S_i\cap\p\mfR^{n+1}_+=\emptyset$, the convention $\theta_i=\pi$ will be used.
It follows immediately that
\eq{
\Abs{\mcW_{\frac14\tilde D_1}(o_i)\cap\overline{\mfR^{n+1}_+}}
=\mathfrak{b}_{\theta_i}(\frac14\tilde D_1)^{n+1}.
}

Note that the difference of the volumes of $\mcW_{\frac14\tilde D_1}(o_i)\cap\mfR^{n+1}_+$ and $\mcW_{\frac14\tilde D_1+\ep_0}(o_i)\cap\mfR^{n+1}_+$ satisfies the estimate
\eq{\label{ineq-volume-difference-balls}
&\Abs{(\mcW_{\frac14\tilde D_1+\ep_0}(o_i)\cap\overline{\mfR^{n+1}_+})\setminus(\mcW_{\frac14\tilde D_1}(o_i)\cap\overline{\mfR^{n+1}_+)}}\\
\leq&\Abs{\mcW_{\frac14\tilde D_1+\ep_0}(o_i)\setminus \mcW_{\frac14\tilde D_1}(o_i)}
\leq C\ep_0
=C\ep^\frac1{2(n+2)},
}
and hence from \eqref{ineq-JN23-(3-28)} we deduce
\eq{
\Abs{(\Om_{r_0}+\mcW_{\frac14\tilde D_1})\cap\overline{\mfR^{n+1}_+}-\sum_{i=1}^N\mathfrak{b}_{\theta_i}(\frac14\tilde D_1)^{n+1}}
\leq C\ep^\frac1{2(n+2)}.
}
On the other hand,
choosing $\rho=\frac14\tilde D_1$ in \eqref{ineq-Om_r-rho}, we obtain
\eq{
\Abs{(\Om_{r_0}+\mcW_{\frac14\tilde D_1})\cap\overline{\mfR^{n+1}_+}-\frac{\abs{\Om}}{R^{n+1}}(\frac14\tilde D_1+\ep^\frac1{n+2})^{n+1}}
\leq C\ep^\frac1{(n+2)}.
}
Combining these estimates, we find that
\eq{\label{ineq-JN23-(3-31)}
\Abs{\abs{\Om}-\sum_{i=1}^N\mathfrak{b}_{\theta_i}R^{n+1}}
\leq C\ep^\frac1{2(n+2)},
}
and hence from \eqref{ineq-JN23-(3-4)} (recall that $\lambda=\frac{n}R$)
\eq{
\Abs{P_F(\Om;\mfR^{n+1}_+)-(n+1)\sum_{i=1}^N\mathfrak{b}_{\theta_i}R^n}
\leq C\ep^\frac1{2(n+2)}.
}

\noindent{\bf Step 2. We locate the centers $\{o_i\}_{i=1,\ldots, N}$.
}

Using \eqref{ineq-JN23-(3-31)}, \eqref{ineq-JN23-3-22} (with $\rho=r=r_0$), \eqref{ineq-JN23-(3-28)} (with $\rho=r_0$), we find
\eq{
\sum_{i=1}^N\mathfrak{b}_{\theta_i}R^{n+1}
\leq&\abs{\Om}+C\ep^\frac1{2(n+2)}
\leq\Abs{(\Om_{r_0}+\mcW_{r_0})\cap\overline{\mfR^{n+1}_+}}+C\ep^\frac1{2(n+2)}\\
\leq&\Abs{\bigcup_{i=1}^N\left(\mcW_{r_0+\ep_0}(o_i)\cap\overline{\mfR^{n+1}_+}\right)}+C\ep^\frac1{2(n+2)}.
}
Let $\tilde\theta_i$ denote the contact angle of $\p \mcW_{r_0}(o_i)$ with $\p\mfR^{n+1}_+$. %and $\theta'_i$ denote the contact angle of $\p B_{\frac12R;\theta}(o_i)$.
Arguing as \eqref{ineq-volume-difference-balls}, we get
\eq{
\sum_{i=1}^N\mathfrak{b}_{\theta_i}R^{n+1}
\leq&\Abs{\bigcup_{i=1}^N\left(\mcW_{r_0}(o_i)\cap\overline{\mfR^{n+1}_+}\right)}+C\ep^\frac1{2(n+2)}
\leq\sum_{i=1}^N\mathfrak{b}_{\tilde\theta_i}R^{n+1}+C\ep^\frac1{2(n+2)},
}
which yields that
\eq{
\sum_{i=1}^N\left(\mathfrak{b}_{\theta_i}-\mathfrak{b}_{\tilde\theta_i}\right)
\leq C\ep^\frac1{2(n+2)}.
}
By virtue of the monotonicity result (Lemma \ref{Lem-monotonicity-theta}), and recall that by definition, $\frac14\tilde D_1<\frac14R<\frac12R<r_0<R$, we thus obtain $\theta\leq\tilde\theta_i\leq\theta_i\leq\pi$, and
\eq{\label{ineq-b_theta-b_theta'}
0
%\leq\sum_{i=1}^N\left(\mathfrak{b}_{\theta_i}-\mathfrak{b}_{\theta'_i}\right)
\leq\sum_{i=1}^N\left(\mathfrak{b}_{\theta_i}-\mathfrak{b}_{\tilde\theta_i}\right)
\leq C\ep^\frac1{2(n+2)}.
}

Applying \eqref{ineq-theta_rho-theta} for $\theta_i$ and $\tilde\theta_i$, if both of them $<\pi$, we obtain
\eq{\label{ineq-cos-tildetheta_i-theta_i}
\cos\tilde\theta_i-\cos\theta_i
=&-\frac{\left<o_i,E_{n+1}\right>}{r_0}
-(-\frac{\left<o_i,E_{n+1}\right>}{\frac14\tilde D_1})\\
=&(\frac4{\tilde D_1}-\frac1{r_0})\left<o_i,E_{n+1}\right>
\geq\frac2R\left<o_i,E_{n+1}\right>.
}
Moreover, exploiting \eqref{eq-b_theta}, we can estimate the location of $o_i$.
Precisely, we consider the following cases separately.

{\bf Case 1}.  $\frac\pi2<\theta
\leq\tilde\theta_i\leq\theta_i\leq\pi$.

Note that in this case, we have $-1\leq\cos\theta_i\leq\cos\tilde\theta_i\leq\cos\theta<0$, and hence
\eq{
\mathfrak{b}_{\theta_i}-\mathfrak{b}_{\tilde\theta_i}
=&C(n)\int_{1-\cos^2\theta_i}^{1-\cos^2\tilde\theta_i}t^{\frac{n}2}(1-t)^{-\frac12}\rd t\\
%Let us first consider the case when $\theta_i<\pi$, it follows that
%\eq{
%\mathfrak{b}_{\theta_i}-\mathfrak{b}_{\tilde\theta_i}
\geq&C(n)\left(1-\cos^2\tilde\theta_i-(1-\cos^2\theta_i)\right)(1-\cos^2\theta_i)^\frac{n}2(1-(1-\cos^2\theta_i))^{-\frac12}\\
\geq&C(n)(\cos\tilde\theta_i-\cos\theta_i)(-\cos\tilde\theta_i-\cos\theta_i)\sin^n\theta_i(-\cos\theta_i)^{-1}\\
\geq& C(n)(-2\cos\theta)(\cos\tilde\theta_i-\cos\theta_i)\sin^n\theta_i,
}
which, in conjunction with \eqref{ineq-b_theta-b_theta'},
shows that
\eq{\label{ineq-cos-sin^n-theta_i}
(\cos\tilde\theta_i-\cos\theta_i)\sin^n\theta_i
\leq C\ep^\frac1{2(n+2)}.
}

If $\sin^n\theta_i\geq\ep^\frac{n}{2(n+2)^2}$, then we immediately deduce from the above inequality that
\eq{
\cos\tilde\theta_i-\cos\theta_i
\leq C\ep^\frac1{(n+2)^2}.
}
Taking \eqref{ineq-cos-tildetheta_i-theta_i} into consideration, we thus deduce that
\eq{
\left<o_i,E_{n+1}\right>\leq C\ep^\frac1{(n+2)^2}.
}

If not, then we have $\sin^n\theta_i<\ep^\frac{n}{2(n+2)^2}$ (that is, $\sin^2\theta_i<\ep^\frac{1}{(n+2)^2}$).
By virtue of \eqref{ineq-b_theta-b_theta'} and \eqref{eq-b_theta}, we find
\eq{
C\ep^\frac1{2(n+2)}
\geq \mathfrak{b}_{\theta_i}-\mathfrak{b}_{\tilde\theta_i}
=&C(n)\int_{\sin^2\theta_i}^{\sin^2\tilde\theta_i}t^{\frac{n}2}(1-t)^{-\frac12}\rd t
>C(n)\int_{\ep^\frac1{(n+2)^2}}^{\sin^2\tilde\theta_i}t^{\frac{n}2}(1-t)^{-\frac12}\rd t\\
\geq&C(n)(\sin^2\tilde\theta_i-\ep^\frac1{(n+2)^2})\ep^\frac{n}{2(n+2)^2}.
}
A direct computation then shows that
\eq{
\sin^2\tilde\theta_i
\leq C\ep^\frac1{(n+2)^2},
}
and equivalently,
%\eq{
%1-C\ep^\frac1{(n+2)^2}
%\leq\cos^2\tilde\theta_i\leq1.
%}
\eq{
-1\leq\cos\tilde\theta_i
\leq-1+C\ep^\frac1{(n+2)^2}.
}
Recalling \eqref{ineq-theta_rho-theta} and the definition of $\tilde\theta_i$, we deduce that
\eq{
\left<o_i,E_{n+1}\right>
\geq (1+\cos\theta)r_0-C\ep^\frac1{(n+2)^2}.
}

%-----

{\bf Case 2.1.} $0<\theta\leq\frac\pi2$, and $\theta\leq\tilde\theta_i\leq\theta_i\leq\frac\pi2$.

In this case, we have $\sin\theta\leq\sin\tilde\theta_i\leq\sin\theta_i\leq1$, $0\leq\cos\theta_i\leq\cos\tilde\theta_i\leq\cos\theta$.
By using \eqref{eq-b_theta}, we find
\eq{
\mathfrak{b}_{\theta_i}-\mathfrak{b}_{\tilde\theta_i}
=&C(n)\int_{1-\cos^2\tilde\theta_i}^{1-\cos^2\theta_i}t^\frac{n}2(1-t)^{-\frac12}\rd t.
}
If $\sin^2\tilde\theta_i=1$, we must have $\sin^2\theta_i=1$ as well, it follows from \eqref{ineq-theta_rho-theta} that this case is only possible if and only if $\theta=\frac\pi2$ and $\left<o_i,E_{n+1}\right>=0$.

Otherwise, we learn from the above estimate that
\eq{
\mathfrak{b}_{\theta_i}-\mathfrak{b}_{\tilde\theta_i}
\geq C(n)(\cos\tilde\theta_i-\cos\theta_i)(\cos\tilde\theta_i+\cos\theta_i)(\cos\tilde\theta_i)^{-1}\sin^n\tilde\theta_i
\geq C(n,\theta)(\cos\tilde\theta_i-\cos\theta_i),
}
which, in conjunction with \eqref{ineq-b_theta-b_theta'} and \eqref{ineq-cos-tildetheta_i-theta_i},
shows that
\eq{
\left<o_i,E_{n+1}\right>
\leq C\ep^\frac1{2(n+2)}.
}

{\bf Case 2.2.} $0<\theta\leq\frac\pi2$, and $\frac\pi2\leq\tilde\theta_i\leq\theta_i\leq\pi$.

In this case, we have $0\leq\sin\theta_i\leq\sin\tilde\theta_i\leq1$, $-1\leq\cos\theta_i\leq\cos\tilde\theta_i\leq0\leq\cos\theta$.
Using \eqref{eq-b_theta}, we find
\eq{
\mathfrak{b}_{\theta_i}-\mathfrak{b}_{\tilde\theta_i}
=&C(n)\int_{1-\cos^2\theta_i}^{1-\cos^2\tilde\theta_i}t^{\frac{n}2}(1-t)^{-\frac12}\rd t.
}
If $\sin\theta_i=1$, then we must have
$\sin\tilde\theta_i=\sin\theta=1$ as well,
it follows from \eqref{ineq-theta_rho-theta} that $\left<o_i,E_{n+1}\right>=0$ and $\theta=\frac\pi2$.

If $\sin\theta_i<1$, we learn from the above estimate that
\eq{
\mathfrak{b}_{\theta_i}-\mathfrak{b}_{\tilde\theta_i}
\geq C(n)(\cos\tilde\theta_i-\cos\theta_i)\frac{\cos\tilde\theta_i+\cos\theta_i}{\cos\theta_i}\sin^n\theta_i
\geq C(n)(\cos\tilde\theta_i-\cos\theta_i)\sin^n\theta_i,
}
which, in conjunction with \eqref{ineq-b_theta-b_theta'},
shows that
\eq{
(\cos\tilde\theta_i-\cos\theta_i)\sin^n\theta_i
\leq C\ep^\frac1{2(n+2)}.
}
The rest of the proof of this case follows from that of {\bf Case 1}.

{\bf Case 2.3.} 
$0<\theta\leq\frac\pi2$, and $0<\theta\leq\tilde\theta_i\leq\frac\pi2\leq\theta_i\leq\pi$.

Using \eqref{eq-b_theta}, we find
\eq{\label{ineq-max-sin-theta_i-theta'_i}
\mathfrak{b}_{\theta_i}-\mathfrak{b}_{\tilde\theta_i}
=&\om_{n+1}-\frac{\om_{n+1}}2\left(I_{\sin^2\theta_i}(\frac{n+2}2,\frac12)+I_{\sin^2\tilde\theta_i}(\frac{n+2}2,\frac12)\right)\\
=&C(n)\left(\int_{\sin^2\theta_i}^1t^{\frac{n}2}(1-t)^{-\frac12}\rd t+\int_{\sin^2\tilde\theta_i}^1t^{\frac{n}2}(1-t)^{-\frac12}\rd t\right)\\
\geq&2C(n)\int^1_{\max\{\sin^2\theta_i,\sin^2\tilde\theta_i\}}t^\frac{n}2(1-t)^{-\frac12}\rd t.
}
If $\max\{\sin^2\theta_i,\sin^2\tilde\theta_i\}=1$, then it is easy to see that either $\min\{\sin^2\theta_i,\sin^2\tilde\theta_i\}=1$ as well, with $\theta=\frac\pi2$; or one of the following situations:

{\bf Case 2.3.1.}
$\theta\leq\tilde\theta_i<\frac\pi2=\theta_i$.

In this case, we first apply \eqref{ineq-theta_rho-theta} on $\theta_i$ to see that
\eq{
\left<o_i,E_{n+1}\right>
=\frac{\tilde D_1}4\cos\theta,
}
and then use \eqref{eq-b_theta} to obtain
\eq{
\mathfrak{b}_{\theta_i}-\mathfrak{b}_{\tilde\theta_i}
=&C(n)\int^1_{\sin^2\tilde\theta_i}t^\frac{n}2(1-t)^{-\frac12}\rd t
\geq C(n)(1-\sin^2\tilde\theta_i)\sin^n\tilde\theta_i(\cos\tilde\theta_i)^{-1}\\
=& C(n)\sin^n\tilde\theta_i\cos\tilde\theta_i
\geq C(n,\theta)\cos\tilde\theta_i.
}
By virtue of \eqref{ineq-b_theta-b_theta'} and again, \eqref{ineq-theta_rho-theta}, we find
\eq{
\frac{\tilde D_1}4\cos\theta
=\left<o_i,E_{n+1}\right>
\geq r_0\cos\theta-C\ep^\frac1{2(n+2)},
}
implying that $0<\cos\theta\leq C\ep^\frac1{2(n+2)}$, and in turn,
\eq{
\left<o_i,E_{n+1}\right>
=\frac{\tilde D_1}4\cos\theta
\leq C\ep^\frac1{2(n+2)}.
}

{\bf Case 2.3.2.} $0<\theta\leq\tilde\theta_i=\frac\pi2<\theta_i\leq\pi$.

We first observe that $\theta_i\leq\frac34\pi$, otherwise
$\mathfrak{b}_{\theta_i}-\mathfrak{b}_{\tilde\theta_i}>\mathfrak{b}_{\frac34\pi}-\mathfrak{b}_{\frac12\pi}$, a contradiction to \eqref{ineq-b_theta-b_theta'}. 
Then we may
argue as {\bf Case 2.3.1} to get
\eq{
\left<o_i,E_{n+1}\right>
=r_0\cos\theta,
}
and also
\eq{
\frac{4r_0-\tilde D_1}{\tilde D_1}\cos\theta
=-\cos\theta+\frac4{\tilde D_1}\left<o_i,E_{n+1}\right>
=-\cos\theta_i
\leq C\ep^\frac1{2(n+2)}.
}
It follows that
\eq{
\cos\theta\leq C\ep^\frac1{2(n+2)},
}
and hence
\eq{
\left<o_i,E_{n+1}\right>
=r_0\cos\theta
\leq C\ep^\frac1{2(n+2)}.
}

It is thus left to consider the situation when $\max\{\sin^2\theta_i,\sin^2\tilde\theta_i\}<1$, namely, $0<\theta\leq\tilde\theta_i<\frac\pi2<\theta_i\leq\pi$.

{\bf Case 2.3.3.} 
$0<\theta\leq\tilde\theta_i<\frac\pi2<\theta_i\leq\pi$, with $\sin\theta_i<\sin\tilde\theta_i<1$.

Using \eqref{ineq-max-sin-theta_i-theta'_i} and taking \eqref{ineq-b_theta-b_theta'} into account, we find:
\eq{
C\ep^\frac1{2(n+2)}
\geq C(n)\sin^n\tilde\theta_i\cos\tilde\theta_i
\geq C(n,\theta)\cos\tilde\theta_i,
}
thus
\eq{
1-C\ep^\frac1{n+2}
\leq\sin^2\tilde\theta_i<1.
}
Substituting this back into \eqref{ineq-max-sin-theta_i-theta'_i}, we get
\eq{
C\ep^\frac1{2(n+2)}
\geq\mathfrak{b}_{\theta_i}-\mathfrak{b}_{\tilde\theta_i}
\geq& C(n)\left(\sin^n\theta_i(-\cos\theta_i)+\int^1_{1-C\ep^\frac1{n+2}}t^\frac{n}2(1-t)^{-\frac12}\rd t\right)\\
\geq&C(n)\left(\sin^n\theta_i(-\cos\theta_i)+C\ep^\frac1{2(n+2)}\right),
}
therefore
\eq{
\sin^n\theta_i(-\cos\theta_i)
\leq C\ep^\frac1{2(n+2)},
}
combined with the estimate that $\cos\tilde\theta_i\leq C\ep^\frac1{2(n+2)}$, we again deduce \eqref{ineq-cos-sin^n-theta_i}.
Then we may follow the proof of {\bf Case 1} to conclude that
\eq{
\left<o_i,E_{n+1}\right>
\leq C\ep^\frac1{(n+2)^2}.
}
Note that the situation $\sin^n\theta_i<\ep^\frac{n}{2(n+2)^2}$ would not happen, because that means $\theta_i$ is close to being $\pi$, which contradicts to the second equality in \eqref{ineq-max-sin-theta_i-theta'_i}, due to the estimate \eqref{ineq-b_theta-b_theta'}.

{\bf Case 2.3.4.} 
$0<\theta\leq\tilde\theta_i<\frac\pi2<\theta_i\leq\pi$, with $\sin\tilde\theta_i\leq\sin\theta_i<1$.

We again use \eqref{ineq-max-sin-theta_i-theta'_i} and \eqref{ineq-b_theta-b_theta'} to obtain
\eq{
\sin^n\theta_i(-\cos\theta_i)
\leq C\ep^\frac1{2(n+2)}.
}
Since $\sin\theta\leq\sin\tilde\theta_i\leq\sin\theta_i$, we deduce from the above estimate that
\eq{
0<-\cos\theta_i\leq C\ep^\frac1{2(n+2)}.
}
The rest of the proof follows from that of {\bf Case 2.3.3}.

To complete this step, we point out that, despite the different powers of $\ep$ we obtain when $\theta\in(0,\frac\pi2]$ or $\theta\in(\frac\pi2,\pi)$,
for readers' convenience, we use the following unified estimate on the locations of $o_1,\ldots,o_N$ for any $\theta\in(0,\pi)$: either
\eq{
\left<o_i,E_{n+1}\right>
\leq C\ep^\frac1{(n+2)^2},
}
or
\eq{\label{ineq-o_i,E_n+1-lowerbound}
\left<o_i,E_{n+1}\right>
\geq (1+\cos\theta)r_0-C\ep^\frac1{(n+2)^2}.
}

\noindent{\bf Step 3. We improve the lower bound of \eqref{ineq-JN23-(3-24)-F^o}.
}

Precisely, we show that for the previously defined $o_i\in\Om_{r_0}^i$ and $o_j\in\Om_{r_0}^j$, there exists a positive constant $C_1=C_1(n,\theta,C_0)$ such that for $\theta\in[\frac\pi2,\pi)$,
\eq{\label{ineq-JN23-(3-29)}
\abs{o_i-o_j}
\geq2R-2C_1\ep^\frac1{(n+2)^2}\text{ whenever }i\neq j,
}
and for $\theta\in(0,\frac\pi2)$, if $o_i$ or $o_j$ satisfies \eqref{ineq-o_i,E_n+1-lowerbound}, then \eqref{ineq-JN23-(3-29)} holds; otherwise, one must have
\eq{
F^o(o_i-o_j)
\geq2R-2C_1\ep^\frac1{(n+2)^2},\text{ whenever }i\neq j.
}

{\bf Case 1.} $\theta\in[\frac\pi2,\pi)$. 

We claim that there exists a positive constant $C_1=C_1(n,\theta,C_0)$, such that if there is some $0<h<\frac12R$ satisfying $\abs{o_i-o_j}<2R-2h$ for some $i\neq j$, then it must be that $h\leq C_1\ep^\frac1{(n+2)^2}$.

If this is true, then it follows immediately that we must have
\eq{
\abs{o_i-o_j}
\geq2R-2C_1\ep^\frac1{(n+2)^2}\text{ whenever }i\neq j.
}

Let us now prove the claim. First, note that there holds
\eq{
\Abs{o_i-R\cos\theta E_{n+1}-(o_j-R\cos\theta E_{n+1})}
=\abs{o_i-o_j}
<2R-2h,
}
which means, the balls $\mcW_{R}(o_i)$ (which is $B_{R;\theta}(o_i)$) and $\mcW_{R}(o_j)$ intersect each other and the intersection contains at least a doubled spherical cap (so-called lens) with height $h$, radius $R$.
On the other hand, since $\theta\in[\frac\pi2,\pi)$, and $o_i,o_j\in\overline{\mfR^{n+1}_+}$, we know that not only the enclosed region of the spherical cap has volume lower bound $C(n)R^{n+1}h^\frac{n+2}2$, but also at least half of it is contained in $\overline{\mfR^{n+1}_+}$; that is to say, we have
\eq{
\Abs{\mcW_{R}(o_i)\cap \mcW_{R}(o_j)\cap\overline{\mfR^{n+1}_+}}
\geq Ch^\frac{n+2}2.
}
This fact, together with \eqref{ineq-JN23-(3-31)}, \eqref{ineq-JN23-3-22} (with $\rho=r=r_0$), \eqref{ineq-JN23-(3-28)} (with $\rho=r_0$), yields
\eq{
\sum_{i=1}^N\mathfrak{b}_{\theta_i}R^{n+1}
\leq&\abs{\Om}+C\ep^\frac1{2(n+2)}
\leq\Abs{(\Om_{r_0}+\mcW_{r_0})\cap\overline{\mfR^{n+1}_+}}+C\ep^\frac1{2(n+2)}\\
\leq&\Abs{\bigcup_{i=1}^N\left(\mcW_{r_0+\ep_0}(o_i)\cap\overline{\mfR^{n+1}_+}\right)}+C\ep^\frac1{2(n+2)}.
}
Arguing as \eqref{ineq-volume-difference-balls}, and using Lemma \ref{Lem-monotonicity-theta}, we get
\eq{
\sum_{i=1}^N\mathfrak{b}_{\theta_i}R^{n+1}
\leq&\Abs{\bigcup_{i=1}^N\left(\mcW_{r_0}(o_i)\cap\overline{\mfR^{n+1}_+}\right)}+C\ep^\frac1{2(n+2)}\\
\leq&\sum_{i=1}^N\mathfrak{b}_{\tilde\theta_i}R^{n+1}-\Abs{\mcW_{R}(o_i)\cap \mcW_{R}(o_j)\cap\overline{\mfR^{n+1}_+}}+C\ep^\frac1{2(n+2)}\\
\leq&\sum_{i=1}^N\mathfrak{b}_{\theta_i}R^{n+1}-Ch^\frac{n+2}2+C\ep^\frac1{2(n+2)},
}
which implies that $h\leq C_1\ep^\frac1{(n+2)^2}$ for some $C_1=C_1(n,\theta,C_0)$ and concludes the case.

{\bf Case 2.} $\theta\in(0,\frac\pi2)$.

We first note that, if \eqref{ineq-o_i,E_n+1-lowerbound} is satisfied for one of $o_i$ and $o_j$, let us say $o_i$, then $\mcW_R(o_i)$ is a Euclidean ball which is almost completely contained in $\overline{\mfR^{n+1}_+}$,
%(namely, $\theta_{i(j)},\tilde\theta_{i(j)}$ are close to $\pi$), then one 
so we may follow the proof of {\bf Case 1} to show that, if there is some $0<h<\frac12R$ satisfying $\abs{o_i-o_j}<2R-2h$ for some $i\neq j$, then it must be that $h\leq C_1\ep^\frac1{(n+2)^2}$.
It is thus left to consider the case when both $o_i$ and $o_j$ satisfy
\eq{\label{ineq-o_i,E_n+1-upperbound}
\left<o_i,E_{n+1}\right>,\left<o_j,E_{n+1}\right>
\leq C\ep^\frac1{(n+2)^2},
}
as observed in {\bf Step 2}.

We claim as well that there exists a positive constant $ C_1= C_1(n,\theta,C_0)$, such that if there is some $0<h<\frac12R$ satisfying $F^o(o_i-o_j)<2R-2h$ for some $i\neq j$, then $h\leq C_1\ep^\frac1{(n+2)^2}$.

First, we conclude from \eqref{ineq-o_i,E_n+1-upperbound} and
the triangle inequality on $F^o$ that
\eq{
F^o(\tilde o_i-\tilde o_j)
<2R-2h+C\ep^\frac1{(n+2)^2},
}
where $\tilde o_i,\tilde o_j$ are projections of $o_i,o_j$ on $\p\mfR^{n+1}_+$.

Since $F^o$ is even on the $n$-dimensional space ${\rm span}\{E_1,\ldots,E_n\}$, we infer that
\eq{
F^o(\tilde o_j-\tilde o_i)
=F^o(\tilde o_i-\tilde o_j)
<2R-2h+C\ep^\frac1{(n+2)^2},
}
which in turn shows that $F^o(o_j-o_i)<2R-2h+C\ep^\frac1{(n+2)^2}$, and hence the intersection of $\mcW_{R}(o_i)\cap\overline{\mfR^{n+1}_+}$ and $\mcW_{R}(o_j)\cap\overline{\mfR^{n+1}_+}$ is non-trivial and has volume bounded from below by $Ch^\frac{n+2}2$.
A similar argument follows from that of {\bf Case 1} then proves the claim and hence finishes the step.

\noindent{\bf Step 4. 
We show that there exists $C_2=C_2(n,\theta,C_0)$, such that for
\eq{
r_i\coloneqq R-C_1\ep^{\frac1{(n+2)^2}} \text{ and } r_e\coloneqq R+C_2\ep^{\frac1{(n+2)^2}},}
there holds
\eq{\label{ineq-r_i-r_e}
\bigcup_{i=1}^N\mcW_{r_i}(o_i)\cap\overline{\mfR^{n+1}_+}
\subset\Om
\subset\bigcup_{i=1}^N\mcW_{r_e}(o_i)\cap\overline{\mfR^{n+1}_+},
}
which readily implies
\eqref{esti-Hausdorff-distance}.
}

To prove the first inclusion, we note that 
after decreasing $\ep$, if needed, we shall have $0<r_i<r_0$.

Exploiting
\eqref{ineq-JN23-(3-28)} and {\bf Claim 3} in the proof of Proposition \ref{Prop-JN23-Prop3.3}, we obtain
\eq{\label{inclusions:interior-ball}
\bigcup_{i=1}^N\left(\mcW_{r_i}(o_i)\cap\overline{\mfR^{n+1}_+}\right)
\subset(\Om_{r_0}+\mcW_{r_i})\cap\overline{\mfR^{n+1}_+}
\subset\Om_{r_0-r_i}
\subset\Om.
}

To prove the second inclusion in \eqref{ineq-r_i-r_e}, we exploit the estimate \eqref{ineq-JN23-(3-23)}.
Note that this estimate already shows that the set of points on $\S$ which do not belong to $(\overline{\Om}_{r_0}+\overline{\mcW_{r_0}})\cap\overline{\mfR^{n+1}_+}$ is almost $\mcH^n$-negligible.
Let us now show that, for any such point, it has to be $\ep^\frac1{(n+2)^2}$-close to $(\overline{\Om}_{r_0}+\overline{\mcW_{r_0}})\cap\overline{\mfR^{n+1}_+}$.
In fact, for any such point, say $x\in\S$, we define
\eq{
r_x
\coloneqq\sup_{r>0}\left\{r:\mcW_{r}(x)\cap\left(\left(\overline{\Om}_{r_0}+\overline{\mcW_{r_0}}\right)\cap\overline{\mfR^{n+1}_+}\right)=\emptyset\right\}.
}
Thereby, \eqref{ineq-JN23-(3-23)} implies that
\eq{
\mcH^n(\S\cap \mcW_{r_x}(x))
\leq C\ep^\frac1{n+2}.
}
On the other hand, we define $\mathfrak{r}_x=\min\{r_x,\frac{\de_{n,\theta}}\lambda\}$ for $\de_{n,\theta}$ resulting from Proposition \ref{Prop-density}.
Note that Proposition \ref{Prop-density} is applicable here, once we further require $\de<\de_{n,\theta}$.
In particular, we obtain
\eq{
\de_{n,\theta}\mathfrak{r}_x^n
\leq\mcH^n(\S\cap \mcW_{\mathfrak{r}_x}(x)),
}
which in turn implies
\eq{
\min\{r_x,\frac{\de_{n,\theta}}\lambda\}
=\mathfrak{r}_x
\leq C\ep^\frac1{n(n+2)}.
}
Observe that $\frac{\de_{n,\theta}}\lambda$ is bounded from below by some constant that depends only on $n,\theta,C_0$, therefore if necessary, we may choose $\ep$ further small, so that %$\mathfrak{r}_x=r_x$, thus showing
\eq{
r_x
=\mathfrak{r}_x
\leq\widetilde C_2\ep^\frac1{n(n+2)}
\leq\widetilde C_2\ep^\frac1{(n+2)^2},
}
as desired.

An immediate consequence of the above estimate is that,
for any such $x$, there exists some $y\in(\overline{\Om}_{r_0}+\overline{\mcW_{r_0}})\cap\overline{\mfR^{n+1}_+}$, such that $F^o(y-x)\leq\widetilde C_2\ep^\frac1{(n+2)^2}$, and hence
\eq{
F^o(x-y)
\leq\frac{M_{F^o}}{m_{F^o}}F^o(y-x)
\leq\frac{1+\abs{\cos\theta}}{1-\abs{\cos\theta}}\widetilde C_2\ep^\frac1{(n+2)^2}
\eqqcolon\widehat C_2\ep^\frac1{(n+2)^2},
}
which, together with
%due to the definition of $r_x$, \eqref{defn-m-M-F^o}, and 
\eqref{ineq-JN23-(3-28)}, yields
\eq{
\S
\subset&(\overline{\Om}_{r_0}+\overline{\mcW_{r_0}})\cap\overline{\mfR^{n+1}_+}+\overline{\mcW_{\widehat C_2\ep^\frac1{(n+2)^2}}}
\subset(\Om_{r_0}+\mcW_{R})\cap\overline{\mfR^{n+1}_+}+\overline{\mcW_{\widehat C_2\ep^\frac1{(n+2)^2}}}\\
\subset&\bigcup_{i=1}^N\left(\mcW_{R+\ep_0}(o_i)\cap\overline{\mfR^{n+1}_+}\right)+\overline{\mcW_{\widehat C_2\ep^\frac1{(n+2)^2}}}
\subset\bigcup_{i=1}^N\mcW_{R+\ep_0+\widehat C_2\ep^\frac1{(n+2)^2}}(o_i),
}
where the last inclusion follows from the triangle inequality for $F^o$.
Finally, since $\S\subset\overline{\mfR^{n+1}_+}$, and $\ep_0=\ep^\frac1{2(n+2)}<\ep^\frac1{(n+2)^2}$, we readily see that
\eq{
\Om\subset\bigcup_{i=1}^N\mcW_{R+ C_2\ep^\frac1{n(n+2)}}(o_i)\cap\overline{\mfR^{n+1}_+},
}
for $C_2\coloneqq2\widehat C_2$. This proves the second inclusion in \eqref{ineq-r_i-r_e}, and of course, \eqref{esti-Hausdorff-distance}, which completes the proof.

\end{proof}

%-------
%\medskip

%\noindent{\bf\large Statements and declarations.}

%Data sharing not applicable to this article as no datasets were generated or analysed during the current study.

%\

%\noindent{\bf Conflict of interest.} On behalf of all authors, the corresponding author states that there is no conflict of
%interest, the authors have no relevant financial or non-financial interests to disclose.

%\

%==========
\appendix

\setcounter{equation}{0}
\renewcommand{\theequation}{\thesection.\arabic{equation}}

\section{Topping-type inequality}\label{App-1}

We follow the notations in Section \ref{Sec-2.1} with the additional assumption that $\Om$ is connected,
in this regard $\S$ is a connected $\theta$-capillary hypersurface in $\overline{\mbR^{n+1}_+}$ with boundary $\p\S\subset\p\mbR^{n+1}_+$.
Let ${\rm dist}_g$ denote the intrinsic distance function on $\S$.
For every $x\in\S$, define
\eq{
V(x,r)
\coloneqq\mcH^n(\S\cap B_r(x))
=\mcH^n\llcorner\S(B_r(x)).
}

Following \cite{Topping08},
we define for $n\geq2, R>0$ the maximal function
\eq{
M(x,R)
\coloneqq\sup_{r\in(0,R]}\frac1n r^{-\frac{1}{n-1}}V(x,r)^{-\frac{n-2}{n-1}}\int_{\S\cap B_r(x)}\abs{H_\S}\rd\mcH^n,
}
and the function that measures the collapsedeness
\eq{
\kappa(x,R)
\coloneqq\inf_{r\in(0,R]}\frac{V(x,r)}{r^n}.
}

The proof of Theorem \ref{Thm-Topping-ineq} is based on the following lemma.

\begin{lemma}\label{Lem-Topping08-1.2}
Under the assumptions of Theorem \ref{Thm-Topping-ineq}, and assume in addition that $n\geq2$. There exists a constant $\de>0$ depends only on $n,\theta$, for any $x\in\S$ and $R>0$, such that at least one of the following statements hold true:
\begin{enumerate}
\item $M(x,R)>\de$;
\item $\kappa(x,R)>\de$.
\end{enumerate}
\end{lemma}
\begin{proof}
We omit the argument $x$ and denote $V(x,r)$ simply by $V(r)$.

For some $\de>0$ to be chosen later, suppose that Case (1) does not happen, namely, $M(x,R)\leq\de$, then we have from the definition of $M$ that for all $r\in(0,R]$:
\eq{
\label{ineq-Topping08-2.1}
\int_{\S\cap B_r(x)}\abs{H_\S}\rd\mcH^n
\leq n\de r^{\frac{1}{n-1}}V(r)^{\frac{n-2}{n-1}}.
}

Following the proof of Proposition \ref{Prop-density} (with $\mcW_r$ therein replaced by $B_r$), especially {\bf Case 1}, we find
\eq{
V(r)^\frac{n-1}{n}
\leq\sigma(n,\theta)(V'(r)+\norm{H_\S}_{L^1(\S\cap B_r(x))}),
}
combined with \eqref{ineq-Topping08-2.1}, this yields
\eq{
V(r)^\frac{n-1}{n}
\leq\sigma(n,\theta)(V'(r)+n\de r^{\frac{1}{n-1}}V(r)^{\frac{n-2}{n-1}}).
}
Rearranging this inequality we obtain
\eq{\label{ineq-Topping08-2.2}
V'(r)+n\de r^{\frac{1}{n-1}}V(r)^\frac{n-2}{n-1}-\frac{1}{\sigma(n,\theta)}V(r)^\frac{n-1}{n}\geq0.
}
Consider on the other hand the function $v(r)\coloneqq\de r^n$, a simple computation yields
\eq{
v'(r)+n\de r^\frac{1}{n-1}v^\frac{n-2}{n-1}-\frac{1}{\sigma(n,\theta)}v^\frac{n-1}{n}
=(n\de+n\de^\frac{2n-3}{n-1}-\frac{1}{\sigma(n,\theta)}\de^\frac{n-1}{n})r^{n-1},
}
and hence by choosing $0<\de<\frac{1}{2}\om_n$ sufficiently small, depending only on $n,\theta$, we shall have (thanks to the fact that $n\geq2$):
\eq{\label{ineq-Topping08-2.3}
v'(r)+n\de r^\frac{1}{n-1}v^\frac{n-2}{n-1}-\frac{1}{\sigma(n,\theta)}v^\frac{n-1}{n}
\leq0.
}

To proceed, notice that $\frac{V(r)}{r^n}\ra\om_n$ as $r\ra0^+$ for $x\in\S\setminus\p\S$ and $\frac{V(r)}{r^n}\ra\frac12\om_n$ as $r\ra0^+$ for $x\in\p\S$. Taking also \eqref{ineq-Topping08-2.2}, \eqref{ineq-Topping08-2.3} into account, we may then use a standard ODE comparison argument to see that $V(r)>v(r)$ for all $r\in(0,R]$, and hence
\eq{
\kappa(x,R)
=\inf_{r\in(0,R]}\frac{V(x,r)}{r^n}>\de,
}
which completes the proof.
\end{proof}

\begin{proof}[Proof of Theorem \ref{Thm-Topping-ineq}]
As $n=1$, the assertion follows easily, thus we assume that $n\geq2$.

Since $\S$ is compact, we may choose $R>0$ sufficiently large so that $\mcH^n(\S)<\de R^n$ for the positive constant $\de$ obtained from Lemma \ref{Lem-Topping08-1.2}.
Accordingly for any $z\in\S$, it must be that $\kappa(z,R)\leq\frac{V(z,R)}{R^n}<\de$ and it follows from Lemma \ref{Lem-Topping08-1.2} that $M(z,R)>\de$.
In particular, we see from the definition of $M$ that for any $z\in\S$, there exists $r=r(z)$ such that
\eq{
\de<\frac1n r^{-\frac{1}{n-1}}V(z,r)^{-\frac{n-2}{n-1}}\int_{\S\cap B_r(z)}\abs{H_\S}\rd\mcH^n
\leq\frac1n r^{-\frac{1}{n-1}}\left(\int_{\S\cap B_r(z)}\abs{H_\S}^{n-1}\right)^\frac{1}{n-1},
}
where we have used the H\"older inequality.
This in turn gives that
\eq{
r(z)
\leq(\frac1n)^{n-1}\de^{1-n}\int_{\S\cap B_{r(z)}(z)}\abs{H_\S}^{n-1}\rd\mcH^n.
}

Let $p,q\in\S$ be any two points such that 
\eq{
d_{\rm ext}(\S)\coloneqq\max_{x,y\in\S}{\rm dist}(x,y)
=\abs{p-q},
}
and denote by $\gamma\subset\S$ any shortest geodesic joining $p$ and $q$.
Following the point-picking argument in \cite{Topping05}*{Lemma 5.2} (see also \cite{Topping08}), we find that:

There exists a countable (possibly finite) set of points $\{z_i\}\subset\gamma$ such that the balls $\{B_{r(z_i)}(z_i)\}$ are disjoint and satisfy
\eq{
\gamma\subset\bigcup_i B_{3r(z_i)(z_i)}.
}
It follows from the triangle inequality that
\eq{
d_{\rm ext}(\S)
=\abs{p-q}
\leq\sum_i6r(z_i).
}
This together with the previous local estimate, implies that
\eq{
{d}_{\rm ext}(\S)
\leq&6\sum_i r(z_i)
\leq6(\frac1n)^{n-1}\de^{1-n}\sum_i\int_{\S\cap B_{r(z_i)}(z_i)}\abs{H_\S}^{n-1}\rd\mcH^n\\
\leq&6(\frac1n)^{n-1}\de^{1-n}\int_\S\abs{H_\S}^{n-1}\rd\mcH^n,%\\
%\coloneqq& C(n,\theta)\int_\S\abs{H}^{n-1}\rd\mcH^n,
}
which proves the assertion.
\end{proof}
\begin{remark}
\normalfont
When $n=2$, by using a clever doubling construction, Miura extends Topping's inequality to surfaces with boundary \cite{Miura22}*{Theorem 1.1}, which, together with \eqref{eq-divf-capillary-2}, readily yields the Topping-type inequality \eqref{ineq-Topping-capillary} for capillary surfaces in $\overline{\mbR^{3}_+}$.
\end{remark}

%==========

%\bibliographystyle{siam}
\bibliography{BibTemplate.bib}

\end{document}